\documentclass[11pt]{article}

\usepackage[T1]{fontenc}
\usepackage{amsfonts}
\usepackage{amsmath}
\usepackage{amssymb}
\usepackage{amsthm}
\usepackage{bbm}
\usepackage{bm}
\usepackage{mathrsfs}
\usepackage{verbatim}
\usepackage{setspace}
\usepackage{color}
\usepackage{pdfsync}

\usepackage{graphicx}
\usepackage{stmaryrd}
\usepackage{float} %
\usepackage{tikz}
\usetikzlibrary{automata, positioning}
\usepackage{tcolorbox}

\theoremstyle{plain}
\newtheorem{theorem}{Theorem}[section]
\newtheorem{proposition}[theorem]{Proposition}
\newtheorem{lemma}[theorem]{Lemma}
\newtheorem{corollary}[theorem]{Corollary}

\theoremstyle{definition}

\newtheorem{remark}[theorem]{Remark}

\newtheorem{example}[theorem]{Example}

\theoremstyle{remark}

\renewenvironment{thebibliography}[1]{%
\begin{oldthebibliography}{#1}%
\setlength{\baselineskip}{1em}
\linespread{.2}
\small
\setlength{\parskip}{0.25ex}%
\setlength{\itemsep}{.20em}%
}%
{%
\end{oldthebibliography}%
}
\newcommand{\eps}{\varepsilon}

\newcommand{\D}{\mathbb{D}}

\renewcommand{\H}{\mathbb{H}}

\newcommand{\N}{\mathbb{N}}

\newcommand{\Q}{\mathbb{Q}}
\newcommand{\R}{\mathbb{R}}

\newcommand{\cB}{\mathcal{B}}

\newcommand{\cD}{\mathcal{D}}

\newcommand{\cM}{\mathcal{M}}

\DeclareMathOperator{\Law}{Law}

\DeclareMathOperator{\Unif}{Unif}

\DeclareMathOperator{\id}{Id}

\DeclareMathOperator*{\argmax}{arg\, max}

\newcommand{\1}{\mathbf{1}}

\newcommand{\shadow}[2]{\theta^{#2}(#1)}

\numberwithin{equation}{section}

\usepackage{enumerate}
\usepackage{dsfont}
\newcommand{\lst}{\preceq_{\mathrm{st}}}
\newcommand{\var}{\mathrm{Var}}
\newcommand{\weakto}{\buildrel \mathrm{w} \over \to}
\newcommand{\notweakto}{\buildrel \mathrm{w} \over \nrightarrow}

\usepackage[pdfborder={0 0 0}, colorlinks=true, allcolors=black]{hyperref}
\hypersetup{
  urlcolor = black,
  pdfauthor = {Marcel Nutz, Ruodu Wang},
  pdfkeywords = {The Directional Optimal Transport},
  pdftitle = {The Directional Optimal Transport},
  pdfsubject = {The Directional Optimal Transport},
  pdfpagemode = UseNone
}

\begin{document}

\title{
\vspace{-3.2em}
 The Directional Optimal Transport\footnote{The authors are indebted to Mathias Beiglb\"ock, Filippo Santambrogio and Julian Schuessler for fruitful discussions that greatly helped this work.}
}
\date{\today}
\author{Marcel Nutz\thanks{Departments of Statistics and Mathematics, Columbia University, New York, USA, \texttt{mnutz@columbia.edu}. Research supported by an Alfred P.\ Sloan Fellowship and NSF Grant DMS-1812661.}
\and
  Ruodu Wang\thanks%
  {Department of Statistics and Actuarial Science,
  University of Waterloo,
  Waterloo,  Canada,
 \texttt{wang@uwaterloo.ca}.
 Research supported by NSERC Grants RGPIN-2018-03823 and RGPAS-2018-522590.
 }} 
\maketitle \vspace{-1em}
\begin{abstract}
We introduce a constrained optimal transport problem where origins $x$ can only be transported to destinations $y\geq x$. Our statistical motivation is to describe the sharp upper bound for the variance of the treatment effect $Y-X$ given marginals when the effect is monotone, or $Y\geq X$. We thus focus on supermodular costs (or submodular rewards) and introduce a coupling $P_{*}$ that is optimal for all such costs and yields the sharp bound. This coupling admits manifold characterizations---geometric, order-theoretic, as optimal transport, through the cdf, and via the transport kernel---that explain its structure and imply useful bounds. When the first marginal is atomless, $P_{*}$ is  concentrated on the graphs of two maps which can be described in terms of the marginals, the second map arising due to the binding constraint.
\end{abstract}

\vspace{0.9em}

{\small

}

\section{Introduction}
\label{se:introduction}

We study a constrained Monge--Kantorovich optimal transport problem between marginal distributions $\mu$ and $\nu$ on the real line where the couplings are required to be ``directional'' in the sense that an origin $x$ can only be transported to destinations $y$ with $y\geq x$. While one can think of several natural transport or matching problems with such a constraint, our initial motivation comes from the statistical analysis of treatment effects. There, one compares a (treated) experiment group of patients with an (untreated) control group. A fundamental problem is that any potential outcome that treated patients would have received without treatment is not observed, and vice versa. While the marginal distributions $\mu$ and $\nu$ of the performance evaluations $X$ and $Y$ of the two groups can be estimated from experiment data, the joint distribution cannot, as the two groups are non-overlapping by design---Neyman noted as early as~1923 (cf.\ \cite{AronowGreenLee.14}) that there are no unbiased or consistent estimators for %
the covariance. The improvement of the performance measure due to treatment, $Y-X$, is known as \emph{treatment effect.} To test the hypothesis of substantial treatment effect, 
it is important to understand bounds on 
 $\var(Y-X)$ or more generally the joint distribution $P$ of $(X,Y)$.   
Crude (yet popular) bounds can be obtained by mapping one group to the extremes of the support of the other. The classical Fr\'echet--Hoeffding (or Hardy--Littlewood) mechanism gives better bounds and is often used in the literature (see, e.g., \cite{AronowGreenLee.14,FanPark.10}, and \cite{RachevRuschendorf.98a,RachevRuschendorf.98b} for mathematical background). The lower bound for $\var(Y-X)$ over all couplings is attained by the comonotone (or Fr\'echet--Hoeffding) coupling. The upper bound over all couplings leads to the antitone coupling, which may be unrealistic in the context of many treatment effects: this coupling corresponds to the idea that the healthiest untreated subject would have become the least healthy patient if treated, and vice versa, which seems exceedingly pessimistic, e.g., in a study on the impact of physical activity on obesity. As proposed in~\cite{Manski.97}, this issue can be alleviated by the assumption of \emph{monotone treatment effect} when suitable, postulating that the treatment effect is nonnegative: $Y\geq X$ means that an untreated individual's performance would not have been worsened by the treatment, and vice versa. Of course, this assumption is only made after verifying that $\nu$ stochastically dominates $\mu$ in the data.
Under the assumption of monotone treatment effect, the sharp upper bound of $\var(Y-X)$ corresponds to a coupling $P_{*}$ that we call optimal directional coupling.\footnote{We prefer ``directional'' over ``monotone'' as the latter terminology often refers to the Fr\'echet--Hoeffding coupling in the transport literature.} More generally, $P_{*}$ yields the sharp upper bound for $E^{P}[g(X,Y)]$ whenever~$g$ is supermodular. The lower bound remains trivial in that it still corresponds to the comonotone coupling (which satisfies $Y\geq X$ in view of the necessary stochastic dominance), whence our focus on the upper bound.

\begin{figure}[tbh]
\begin{center}
\resizebox{\textwidth}{!}{

\tikzset{every picture/.style={line width=0.75pt}} %

\begin{tikzpicture}[x=0.75pt,y=0.75pt,yscale=-1,xscale=1]
\draw    (397.23,114.37) -- (634.5,113.61) ;

\draw    (397.23,219.68) -- (634.5,218.92) ;

\draw [fill={rgb, 255:red, 255; green, 255; blue, 255 }  ,fill opacity=1 ]   (522.69,113.61) -- (417.91,219.68) ;
\draw [shift={(417.91,219.68)}, rotate = 134.65] [color={rgb, 255:red, 0; green, 0; blue, 0 }  ][fill={rgb, 255:red, 0; green, 0; blue, 0 }  ][line width=0.75]      (0, 0) circle [x radius= 3.35, y radius= 3.35]   ;
\draw [shift={(522.69,113.61)}, rotate = 134.65] [color={rgb, 255:red, 0; green, 0; blue, 0 }  ][fill={rgb, 255:red, 0; green, 0; blue, 0 }  ][line width=0.75]      (0, 0) circle [x radius= 3.35, y radius= 3.35]   ;
\draw [fill={rgb, 255:red, 255; green, 255; blue, 255 }  ,fill opacity=1 ]   (615.36,113.61) -- (486.34,219.68) ;
\draw [shift={(486.34,219.68)}, rotate = 140.57] [color={rgb, 255:red, 0; green, 0; blue, 0 }  ][fill={rgb, 255:red, 0; green, 0; blue, 0 }  ][line width=0.75]      (0, 0) circle [x radius= 3.35, y radius= 3.35]   ;
\draw [shift={(615.36,113.61)}, rotate = 140.57] [color={rgb, 255:red, 0; green, 0; blue, 0 }  ][fill={rgb, 255:red, 0; green, 0; blue, 0 }  ][line width=0.75]      (0, 0) circle [x radius= 3.35, y radius= 3.35]   ;
\draw  [dash pattern={on 0.84pt off 2.51pt}]  (615.36,113.61) -- (417.91,219.68) ;

\draw  [dash pattern={on 0.84pt off 2.51pt}]  (486.34,219.68) -- (522.69,113.61) ;

\draw    (112.5,115.13) -- (373.08,114.37) ;

\draw    (112.5,220.45) -- (373.08,219.68) ;

\draw [fill={rgb, 255:red, 255; green, 255; blue, 255 }  ,fill opacity=1 ]   (323.13,114.37) -- (236,220.07) ;
\draw [shift={(236,220.07)}, rotate = 129.5] [color={rgb, 255:red, 0; green, 0; blue, 0 }  ][fill={rgb, 255:red, 0; green, 0; blue, 0 }  ][line width=0.75]      (0, 0) circle [x radius= 3.35, y radius= 3.35]   ;
\draw [shift={(323.13,114.37)}, rotate = 129.5] [color={rgb, 255:red, 0; green, 0; blue, 0 }  ][fill={rgb, 255:red, 0; green, 0; blue, 0 }  ][line width=0.75]      (0, 0) circle [x radius= 3.35, y radius= 3.35]   ;
\draw [fill={rgb, 255:red, 255; green, 255; blue, 255 }  ,fill opacity=1 ]   (295.12,114.18) -- (264.01,220.26) ;
\draw [shift={(264.01,220.26)}, rotate = 106.35] [color={rgb, 255:red, 0; green, 0; blue, 0 }  ][fill={rgb, 255:red, 0; green, 0; blue, 0 }  ][line width=0.75]      (0, 0) circle [x radius= 3.35, y radius= 3.35]   ;
\draw [shift={(295.12,114.18)}, rotate = 106.35] [color={rgb, 255:red, 0; green, 0; blue, 0 }  ][fill={rgb, 255:red, 0; green, 0; blue, 0 }  ][line width=0.75]      (0, 0) circle [x radius= 3.35, y radius= 3.35]   ;
\draw [fill={rgb, 255:red, 255; green, 255; blue, 255 }  ,fill opacity=1 ]   (347.28,114.37) -- (139.97,220.26) ;
\draw [shift={(139.97,220.26)}, rotate = 152.94] [color={rgb, 255:red, 0; green, 0; blue, 0 }  ][fill={rgb, 255:red, 0; green, 0; blue, 0 }  ][line width=0.75]      (0, 0) circle [x radius= 3.35, y radius= 3.35]   ;
\draw [shift={(347.28,114.37)}, rotate = 152.94] [color={rgb, 255:red, 0; green, 0; blue, 0 }  ][fill={rgb, 255:red, 0; green, 0; blue, 0 }  ][line width=0.75]      (0, 0) circle [x radius= 3.35, y radius= 3.35]   ;
\draw [fill={rgb, 255:red, 255; green, 255; blue, 255 }  ,fill opacity=1 ]   (230.01,114.37) -- (206.75,165.11) -- (181.55,220.07) ;
\draw [shift={(181.55,220.07)}, rotate = 114.63] [color={rgb, 255:red, 0; green, 0; blue, 0 }  ][fill={rgb, 255:red, 0; green, 0; blue, 0 }  ][line width=0.75]      (0, 0) circle [x radius= 3.35, y radius= 3.35]   ;
\draw [shift={(230.01,114.37)}, rotate = 114.63] [color={rgb, 255:red, 0; green, 0; blue, 0 }  ][fill={rgb, 255:red, 0; green, 0; blue, 0 }  ][line width=0.75]      (0, 0) circle [x radius= 3.35, y radius= 3.35]   ;
\draw [fill={rgb, 255:red, 255; green, 255; blue, 255 }  ,fill opacity=1 ]   (229.18,116.19) -- (206.75,165.11) -- (181.55,220.07) ;

\draw [shift={(230.01,114.37)}, rotate = 114.63] [color={rgb, 255:red, 0; green, 0; blue, 0 }  ][line width=0.75]    (10.93,-4.9) .. controls (6.95,-2.3) and (3.31,-0.67) .. (0,0) .. controls (3.31,0.67) and (6.95,2.3) .. (10.93,4.9)   ;
\draw [fill={rgb, 255:red, 255; green, 255; blue, 255 }  ,fill opacity=1 ]   (345.49,115.28) -- (139.97,220.26) ;

\draw [shift={(347.28,114.37)}, rotate = 152.94] [color={rgb, 255:red, 0; green, 0; blue, 0 }  ][line width=0.75]    (10.93,-4.9) .. controls (6.95,-2.3) and (3.31,-0.67) .. (0,0) .. controls (3.31,0.67) and (6.95,2.3) .. (10.93,4.9)   ;
\draw [fill={rgb, 255:red, 255; green, 255; blue, 255 }  ,fill opacity=1 ]   (321.86,115.91) -- (236,220.07) ;

\draw [shift={(323.13,114.37)}, rotate = 129.5] [color={rgb, 255:red, 0; green, 0; blue, 0 }  ][line width=0.75]    (10.93,-4.9) .. controls (6.95,-2.3) and (3.31,-0.67) .. (0,0) .. controls (3.31,0.67) and (6.95,2.3) .. (10.93,4.9)   ;
\draw [fill={rgb, 255:red, 255; green, 255; blue, 255 }  ,fill opacity=1 ]   (294.56,116.1) -- (264.01,220.26) ;

\draw [shift={(295.12,114.18)}, rotate = 106.35] [color={rgb, 255:red, 0; green, 0; blue, 0 }  ][line width=0.75]    (10.93,-4.9) .. controls (6.95,-2.3) and (3.31,-0.67) .. (0,0) .. controls (3.31,0.67) and (6.95,2.3) .. (10.93,4.9)   ;
\draw [fill={rgb, 255:red, 255; green, 255; blue, 255 }  ,fill opacity=1 ]   (521.29,115.03) -- (417.91,219.68) ;

\draw [shift={(522.69,113.61)}, rotate = 134.65] [color={rgb, 255:red, 0; green, 0; blue, 0 }  ][line width=0.75]    (10.93,-4.9) .. controls (6.95,-2.3) and (3.31,-0.67) .. (0,0) .. controls (3.31,0.67) and (6.95,2.3) .. (10.93,4.9)   ;
\draw [fill={rgb, 255:red, 255; green, 255; blue, 255 }  ,fill opacity=1 ]   (613.82,114.88) -- (486.34,219.68) ;

\draw [shift={(615.36,113.61)}, rotate = 140.57] [color={rgb, 255:red, 0; green, 0; blue, 0 }  ][line width=0.75]    (10.93,-4.9) .. controls (6.95,-2.3) and (3.31,-0.67) .. (0,0) .. controls (3.31,0.67) and (6.95,2.3) .. (10.93,4.9)   ;

\draw (416.84,229.61) node    {$x$};
\draw (485.27,230.37) node    {$x'$};
\draw (525.19,103.16) node    {$y$};
\draw (617.86,100.63) node    {$y'$};
\draw (265.27,231.84) node    {$x_{1}$};
\draw (236.96,231.84) node    {$x_{2}$};
\draw (138.72,231.84) node    {$x_{4}$};
\draw (180.35,231.84) node    {$x_{3}$};
\draw (350.19,103.16) node    {$y_{1}$};
\draw (324.19,103.16) node    {$y_{2}$};
\draw (296.19,103.16) node    {$y_{3}$};
\draw (231.19,103.16) node    {$y_{4}$};

\end{tikzpicture}

} 
\end{center}
\vspace{-1em}\caption{Left panel: An example of $P_{*}$, with the $y$-axis shown at the top. Right panel: An improvable pair which can be ``improved'' to the dotted pair.} 
\label{fi:improvable}
\end{figure}
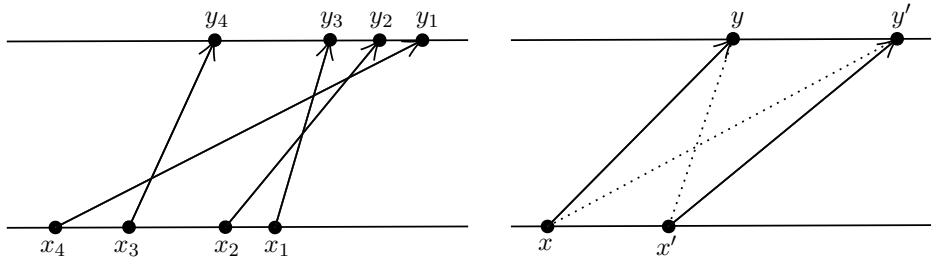

In the next section we introduce $P_{*}$ for general marginals $\mu,\nu$ in stochastic order and provide manifold characterizations that resemble familiar properties of the antitone coupling while also taking into account the constraint. Globally, the geometry is significantly richer than in the classical antitone case. At a local level, the interaction between supermodularity and constraint is much more transparent, and each of our characterizations clarifies that interaction from a different angle.

The construction of $P_{*}$ is best explained in the simple case $\mu=\frac{1}{n} \sum_{i=1}^{n} \delta_{x_{i}}$ and $\nu=\frac{1}{n} \sum_{i=1}^{n} \delta_{y_{i}}$ where both marginals consist of a common number of atoms of equal size at distinct locations, and moreover $x_{1}>\cdots>x_{n}$ are numbered from \emph{right to left}. The transport $P_{*}$ processes these atoms $x_{i}$ in that order, sending each origin to the minimal (left-most) destination $y=T(x_{i})$ that is allowed by the constraint $y\geq T(x_{i})$ and has not been filled yet (Figure~\ref{fi:improvable}). That is, starting with the set $S_{1}=\{y_{1},\dots,y_{n}\}$ of all destinations, we iterate for $k=1,\dots,n$:
\begin{enumerate} \item
$T(x_{k}):= \min \{y\in S_{k}:\, y\geq x_{k}\}$,
\item $S_{k+1}:= S_{k}\setminus \{T(x_k)\}$.
\end{enumerate} 
A less formal description is to imagine a left parenthesis ``$($'' at each location~$x_i$ and a right parenthesis ``$)$'' at each $y_i$. Then $T$ agrees with to the usual rule of matching a left with its corresponding right parenthesis in a mathematical statement.
The antitone coupling would be obtained omitting the inequality in~(i) above, making apparent how the constraint creates the difference with the classical coupling at the local level.

Further properties provided in the next section include a geometric characterization through the support of the coupling and of course the optimality as transport for all supermodular costs (or submodular rewards, including variance of treatment effect); here the notion of cyclical monotonicity plays a key role. In particular, we provide sharp conditions under which $P_{*}$ admits a Monge map. Finally, one can also describe $P_{*}$ through its joint cdf.

The constraint is responsible for qualitative differences with the antitone coupling. Assuming that the first marginal is atomless, the latter coupling always admits a Monge map, %
in other words, it is concentrated on a graph. By contrast, the constrained coupling is concentrated on two graphs. The two maps can be described in detail: one is the identity function and appears when the constraint is locally binding, the other admits a graphical interpretation and a semi-explicit formula based on the difference of the marginal cdf's. The appearance of the identity is clearly reminiscent of the unconstrained transport problem for costs like $c(x,y)=|y-x|^{p},$ $0<p<1$ that combine concavity away from the origin with convexity at the origin, and was first observed in \cite{GangboMcCann.96} in that context. See also \cite[Section~3.3.2]{Santambrogio.15} for a discussion.
Another difference is the behavior under marginal transformations. The antitone coupling is invariant with respect to arbitrary monotone transformations of the coordinate axes; more precisely, the copula corresponding to the coupling is the same for all marginals. This is no longer true for the constrained version, the reason being that the underlying constraint $Y\geq X$ is not invariant. Instead, the copula depends on the marginals and an invariance property holds only when a common transformation is applied to both axes.

Several constrained optimal transport problems have been of lively interest in recent years. One related problem is the optimal transport with quadratic cost $c(x,y)=|y-x|^{2}$ in~$\R^{d}$ studied in \cite{JimenezSantambrogio.12} (see also \cite{CarlierDePascaleSantambrogio.10,ChenJiangYang.13}) under a convex constraint: transports have to satisfy $y-x\in C$ for a given convex set~$C$. It is shown that this problem admits an optimal transport map (Monge map) in great generality. The specification $y-x\in C$ accommodates our constraint, but minimizing the quadratic cost (rather than maximizing) yields the comonotone coupling in our setting. Indeed, \cite{JimenezSantambrogio.12} details that the comonotone coupling is the optimal solution for general~$C$ in the scalar case---the constraint is not binding as soon as an admissible coupling exists. In our problem, the constraint is typically binding and the optimal coupling typically does not admit a Monge map but instead requires a randomization between two maps. (See also Section~\ref{se:moreConstraints} for a generalization of $P_{*}$ to cone constraints that may simplify the comparison with~\cite{JimenezSantambrogio.12}.)

A different constrained problem is the martingale optimal transport introduced in \cite{BeiglbockHenryLaborderePenkner.11, GalichonHenryLabordereTouzi.11,TanTouzi.11}, corresponding to the constraint $E[Y|X]=X$ as motivated from financial mathematics (see \cite{AcciaioBeiglbockPenknerSchachermayer.12,BeiglbockJuillet.12,BiaginiBouchardKardarasNutz.14,BouchardNutz.13,DolinskySoner.12,Hobson.11}, among many others). In particular, the Left- and Right-Curtain couplings of~\cite{BeiglbockJuillet.12} correspond to the constrained versions of the comonotone/antitone couplings. It is worth noting that these couplings are also concentrated on the graphs of two maps in typical cases, like $P_{*}$. (However, the appearance of a randomization is more obvious: only a constant martingale is deterministic.) The supermartingale constraint $E[Y|X]\leq X$ in \cite{NutzStebegg.16} resembles the current situation in being an inequality constraint. Compared to all of these examples, the present case yields by far the most explicit and detailed results. In hindsight, the directional transport is arguably the most canonical and simplest nontrivial example of a constrained optimal transport problem. For general transport problems in Polish spaces, cyclical monotonicity and duality theory with constraints (or equivalently cost functions with infinite values) were studied by \cite{AmbrosioPratelli.03, BeiglbockGoldsternMareschSchachermayer.09, EkrenSoner.18, Kellerer.84,SchachermayerTeichmann.09}, among others.

The literature on copulas features several directly related results; these works seem to be mostly unaware of one another and of the results in the optimal transport literature. The earliest related contribution that we are aware of, \cite{Smith.83}, features a bound on the cdf of any directional coupling (see also Remark~\ref{rk:earlierResults} below). It is not investigated if or when that bound corresponds to a coupling. Almost two decades later, \cite{Rogers.99} was interested in coupling random walks ``fast'' and determined a directional coupling which maximizes a cost of the form $\varphi(y-x)$ with $\varphi$ strictly convex, nonnegative and decreasing. It is clear from Theorem~\ref{th:main} below that this coupling is~$P_{*}$; the decrease of $\varphi$ is irrelevant as convexity alone implies submodularity. In \cite{Rogers.99}, the application to random walks is successful only when the difference of the marginal distributions is unimodular, and in that case, $P_{*}$ has a trivial structure as the sum of an identity and an antitone coupling between disjoint intervals (see Example~\ref{ex:single-crossing} below)---that may explain why \cite{Rogers.99} did not investigate the coupling further. The recent work~\cite{ArnoldMolchanovZiegel.20} characterizes all directional dependence structures of marginals in stochastic order and derives several related bounds, in particular one on the cdf which gives exactly the cdf of~$P_{*}$. (In fact, the same cdf was previously stated in \cite{Rogers.99}, in a slightly more implicit form.) The structure of the coupling, and more generally the point of view of optimal transport, are not highlighted in these works.

While we hope that this paper is a fairly complete study of the scalar case with inequality constraint (or, more generally, one-dimensional cone constraint; cf.\ Section~\ref{se:moreConstraints}), we mention that the multidimensional case is wide open. To stick with the above motivation, consider a treatment which affects two (or more) separately measured qualities---e.g., the impact of physical exercise on blood pressure and body mass index. Control and experiment group now give rise to distributions in $\R^{2}$, and the assumption of monotone treatment effect for both performance measures corresponds to a \emph{cone} constraint $y-x \in [0,\infty)^2$. It is worth noting that even if a scalar quantity is used to aggregate the two performances, the cone constraint is typically more stringent than what would be obtained by constraining the aggregated performances.

The remainder of the paper is organized as follows. Section~\ref{se:mainResults} formalizes the problem and presents the main results. The subsequent Sections~\ref{se:characterizations}--\ref{se:transportMap} provide the proofs and some required tools, as well as examples and additional consequences. Section~\ref{se:furtherProperties} gathers three discussions that we omitted in the main results: another decomposition of $P_{*}$, optimality properties in unconstrained transport problems, and an extension to general (random) cone constraints. 
\section{Main Results}\label{se:mainResults}

Let $\mu$ and $\nu$ be probability measures on $\R$ and denote by $X(x,y)=x$, $Y(x,y)=y$ the coordinate projections on $\R^{2}$. A coupling, or transport, of $\mu$ and $\nu$ is a probability $P$ on $\R^{2}$ with marginals $P\circ X^{-1}=\mu$ and $P\circ Y^{-1}=\nu$. We call a coupling $P$ \emph{directional} if it is concentrated on the closed halfplane above the diagonal,
$$
  \H=\{Y\geq X\}=\{(x,y)\in\R^{2}:\, y\geq x\},
$$
meaning that $\mu$-almost every origin $x$ is transported to a destination located to the right of~$x$ (or to $x$ itself). Denoting by $\cD=\cD(\mu,\nu)$ the set of all directional couplings, we have $\cD\neq\emptyset$ if and only if $\mu$ and $\nu$ are in stochastic order, denoted $\mu \lst \nu$, meaning that their cdf's satisfy $F_{\mu} \ge F_{\nu}$. 
Indeed, $\mu \lst \nu$ if and only if the comonotone coupling is directional. More generally, we indicate by $\theta_1\lst \theta_2$ two subprobabilities with common mass $\theta_{1}(\R)=\theta_{2}(\R)$ and $F_{\theta_1} \ge F_{\theta_2}$. The other notions also have obvious generalizations.

The following theorem corresponds to a general version of the discrete construction of $P_{*}$ in the Introduction. We write
$\theta  \le \nu$ for a subprobability~$\theta$ with $\theta(A)  \le \nu(A)$ for all $A\in\cB(\R)$.

\begin{theorem}\label{th:existenceAndMinimality}
  Let $\mu\lst\nu$. There exists a unique directional coupling $P_{*}=P_{*}(\mu,\nu)$ which couples $\mu|_{(x,\infty)}$ to $\nu_x$ for all $x\in\R$,
  where the subprobability $\nu_x$ is defined by its
   cdf 
   $$F_{\nu_x} = \sup_{\theta \in S_{x}} F_{\theta} \quad \mbox{for}\quad S_{x}=\{\theta: \mu|_{(x,\infty)} \lst \theta  \le \nu \}.%
   $$  
  The measure $\nu_x$ is the unique minimal element of $S_{x}$ for the order $\lst$.
\end{theorem}

The coupling $P_{*}$ differs from the antitone coupling except in the trivial case where all couplings are directional; that is, when $\mu((-\infty,x])=\nu([x,\infty))=1$ for some $x\in\R$. Indeed, this is the only case where the antitone coupling is directional.

We make $\mu\lst\nu$ a \emph{standing assumption} in all that follows. The above theorem is one of several equivalent characterizations of $P_{*}$ that we detail next. The most important for our analysis is geometric, describing the  support of $P_{*}$ based on the idea that we would like any two trajectories of the transport to cross whenever that is allowed by the constraint. We say that the pair $((x,y), (x',y'))\in \H^2$   is \emph{improvable} if $x<x'\leq y<y'$.  
This means that $(x,y)$ and $(x',y')$ do not cross, but they could be rearranged (``improved'') into the configuration 
$((x,y'), (x',y))$ which forms a cross and remains $\H^{2}$ (Figure~\ref{fi:improvable}). 
A set $\Gamma\subseteq \H$ satisfies the \emph{constrained crossing property} if it contains no improvable pairs. Stated differently, any two trajectories in $\Gamma$ either cross, or they cannot be rearranged into a cross without exiting $\H$.

This property is closely related to a characterization of $P_{*}$ through optimal transport with specific reward functions. A Borel function $g:\H\to \R$ is \emph{submodular} (on $\H$) if
\begin{equation}\label{eq:defsub}
g(x,y)+g(x',y') \le g(x,y') + g(x',y) \quad\mbox{for all}\quad x<x'\leq y<y'
\end{equation}
and \emph{strictly submodular} if the inequality in \eqref{eq:defsub} is strict; two examples are $g(x,y)=(x-y)^{2}$ and $g(x,y)=-\sqrt{|x-y|}$. If $g$ is differentiable, the Spence--Mirrlees condition $-g_{xy}>0$ is a sufficient condition. 
We say that $g$ is \emph{$(\mu,\nu)$-integrable} if 
$|g(x,y)|\le \phi(x) + \psi(y)$ for some $\phi \in L^1(\mu)$ and $\psi \in L^1(\nu)$. This implies uniform bounds on $\int g \,dP$ for any coupling $P$ and in particular that the optimal transport problem
\begin{equation}
\sup_{P\in \cD} \int g   \,dP ~~~\left(\mbox{or equivalently,}~ \inf_{P\in \cD} \int -g   \,dP \right)
\label{eq:transportProblem}
\end{equation}
is finite as soon as $\cD\neq \emptyset$.
Finally, $P\in\cD$ is \emph{optimal} for $g$ if it attains the supremum. To see the connection with the  constrained crossing property, observe that for any strictly submodular $g$,
$$
  g(x,y)+g(x',y') < g(x,y')+g(x',y)  \mbox{~~~if~~~$((x,y), (x',y'))$ is improvable.}
$$ 

The following result also contains a third (straightforward) characterization in terms of the so-called concordance order in~(i).

\begin{theorem}\label{th:main}
  For a coupling $P\in \cD(\mu,\nu)$, the following are equivalent.
  \begin{enumerate}
  \item $F_P  \le F_{Q} $ on $\R^2$ for all $Q\in   \cD(\mu,\nu)$,
  where $F_{Q}$ is the cdf of $Q$.
  \item $P$ is optimal for all $(\mu,\nu)$-integrable and submodular $g$.
  \item $P$ is optimal for some $(\mu,\nu)$-integrable and strictly submodular $g$.
  \item $P$ is supported by a set $\Gamma\subseteq \H$ with the constrained crossing property.
  \item $P=P_{*}$. 
  \end{enumerate}   
\end{theorem} 

The geometric characterization in Theorem~\ref{th:main}\,(iv) implies that the optimal coupling $P_*$ is invariant with respect to common transformations of both coordinate axes as follows.

\begin{corollary}\label{co:invariance}
Let $\phi:\R\to\R$ be a strictly increasing function.
Then
$$
  P_*(\mu,\nu) = P_*(\mu\circ \phi^{-1},\nu\circ \phi^{-1}) \circ ( \phi,\phi).
$$
\end{corollary}

In particular, copulas of 
 $P_*(\mu,\nu)$ are precisely those of $P_*(\mu\circ \phi^{-1},\nu\circ \phi^{-1})$, and thus these copulas are invariant under common, strictly increasing transformations of the axes.
 The strict increase of $\phi$ is necessary to retain the constrained crossing property. Similarly, it is clear that the same transformation must be applied to both axes---in contrast to the unconstrained transport problem, as highlighted in the Introduction. %
 
Theorem~\ref{th:main}\,(i) yields an implicit description of the optimal cdf which, by a result of~\cite{ArnoldMolchanovZiegel.20}, implies the following representation. A proof by direct computation will be sketched in Section~\ref{se:cdf}, as well as resulting bounds.

\begin{corollary}\label{co:cdf}
The cdf of $P_{*}$ is given by
\begin{equation}\label{eq:cdf}
 F_{*}(x,y) = \begin{cases}
F_{\nu}(y) &  \mbox{if}~~ y\le x, \\
F_{\mu}(x) - \inf_{z\in [x,y]}  (F_{\mu}(z) -F_{\nu}(z) ) &  \mbox{if}~~ y> x.
\end{cases}
\end{equation}
\end{corollary}

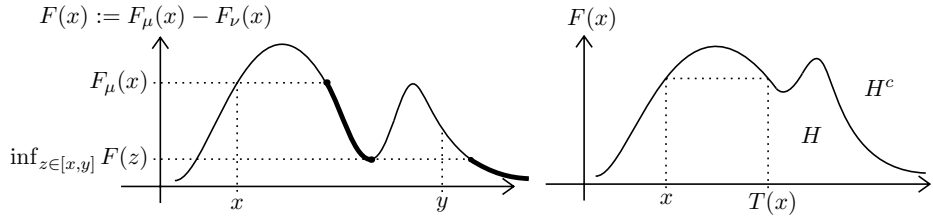
\begin{figure}[b]
\begin{center}
\vspace{-3em}
\resizebox{\textwidth}{!}{

\tikzset{every picture/.style={line width=0.75pt}} %

\begin{tikzpicture}[x=0.75pt,y=0.75pt,yscale=-1,xscale=1]
\draw  (84.38,189.43) -- (328.2,189.43)(108.76,96.84) -- (108.76,199.72) (321.2,184.43) -- (328.2,189.43) -- (321.2,194.43) (103.76,103.84) -- (108.76,96.84) -- (113.76,103.84)  ;
\draw    (117.8,185.77) .. controls (141.63,186.35) and (171.85,16.63) .. (225.91,156.13) ;

\draw    (253.81,148.57) .. controls (275.31,81.73) and (264.27,182.28) .. (337.5,184.6) ;

\draw  [dash pattern={on 0.84pt off 2.51pt}]  (156.45,124.59) -- (156.45,194.42) ;

\draw  [dash pattern={on 0.84pt off 2.51pt}]  (104.58,125.03) -- (211.96,125.03) ;

\draw    (225.91,156.13) .. controls (241.02,193.32) and (249.74,157.87) .. (253.81,148.57) ;

\draw  [dash pattern={on 0.84pt off 2.51pt}]  (104.58,172.46) -- (300.88,172.46) ;

\draw  [dash pattern={on 0.84pt off 2.51pt}]  (283.74,156) -- (283.74,194) ;

\draw [line width=2.25]    (211.96,125.03) .. controls (222.11,142.43) and (228.92,174.14) .. (240.32,172.69) ;

\draw  [color={rgb, 255:red, 0; green, 0; blue, 0 }  ][line width=3] [line join = round][line cap = round] (211.87,124.86) .. controls (211.87,124.86) and (211.87,124.86) .. (211.87,124.86) ;
\draw  [color={rgb, 255:red, 0; green, 0; blue, 0 }  ][line width=3] [line join = round][line cap = round] (239.54,172.75) .. controls (239.54,172.75) and (239.54,172.75) .. (239.54,172.75) ;
\draw [line width=2.25]    (300.88,172.46) .. controls (314.84,181.43) and (325.51,184.04) .. (337.5,184.6) ;

\draw  [color={rgb, 255:red, 0; green, 0; blue, 0 }  ][line width=2.25] [line join = round][line cap = round] (301.18,172.61) .. controls (301.18,172.61) and (301.18,172.61) .. (301.18,172.61) ;
\draw  (348,187.97) -- (586.5,187.97)(371.85,96) -- (371.85,198.19) (579.5,182.97) -- (586.5,187.97) -- (579.5,192.97) (366.85,103) -- (371.85,96) -- (376.85,103)  ;
\draw    (379.01,183.02) .. controls (401.13,183.65) and (429.72,46.47) .. (491.21,128.65) ;

\draw    (507.39,119.17) .. controls (530.04,79.98) and (516.02,178.59) .. (583.98,181.12) ;

\draw  [dash pattern={on 0.84pt off 2.51pt}]  (422.71,122.33) -- (422.71,191.24) ;

\draw  [dash pattern={on 0.84pt off 2.51pt}]  (486.08,122.33) -- (486.08,191.24) ;

\draw  [dash pattern={on 0.84pt off 2.51pt}]  (422.71,122.33) -- (485.81,122.33) ;

\draw    (491.21,128.65) .. controls (496.06,133.71) and (501.46,129.28) .. (507.39,119.17) ;

\draw (156.45,200.45) node    {$x$};
\draw (83.06,124.89) node    {$F_{\mu }( x)$};
\draw (283.74,200.45) node    {$y$};
\draw (58.53,172.94) node    {$\inf_{z\in [ x,y]} F( z)$};
\draw (105.99,84.28) node    {$F( x) :=F_{\mu }( x) -F_{\nu }( x)$};
\draw (422.98,197.4) node    {$x$};
\draw (489.08,199.66) node    {$T( x)$};
\draw (513.05,157.73) node    {$H$};
\draw (555.13,128.02) node    {$H^{c}$};
\draw (376.99,85.28) node    {$F( x)$};
\end{tikzpicture}

}
\end{center}
\vspace{-1.5em}\caption{Left panel: On the formula for $F_{*}$. Right panel: Definition of $T$.}\vspace{-1em}
\label{fi:cdfandT}
\end{figure}

See also Figure~\ref{fi:cdfandT} for a graphical representation.
As a first consequence, we observe the continuity of $P_{*}$ with respect to weak convergence ($\weakto$) of the marginals. 

\begin{corollary}\label{co:continuityWrtMarginals}
  Consider marginals $\mu_n\lst \nu_n$, $n\geq1$ with $\mu_n\weakto\mu$ and $\nu_n\weakto \nu$, and suppose that $\mu$ and $\nu$ are atomless. Then $P_*(\mu_n,\nu_n)\weakto P_*(\mu,\nu)$.
\end{corollary}

We will see in Example~\ref{ex:non-convergence} that the continuity can fail in the presence of atoms.

The subsequent results describe the finer structure of the optimal transport. The \emph{common part} $\mu \wedge\nu$ of $\mu$ and $\nu$ is the measure defined by
$$
\frac{d(\mu \wedge\nu)}{d(\mu+\nu)} := \frac{d\mu}{d(\mu+\nu)} \wedge \frac{du}{d(\mu+\nu)}.
$$
Alternately, $\mu \wedge\nu$ is the maximal measure $\theta$ satisfying $\theta\leq\mu$ and $\theta\leq\nu$, and we can note that $\mu,\nu$ are mutually singular if and only if $\mu \wedge\nu=0$. Importantly, $P_{*}$ always transports $\mu \wedge\nu$ according to the identity coupling, similarly as in~\cite[Main Theorem~6.4]{GangboMcCann.96} for unconstrained transport with cost $l(|y-x|)$ and~$l$ strictly concave (see Figure~\ref{fi:identity} for two simple examples).

\begin{proposition}\label{pr:decompose}
  The optimal coupling $P_{*}(\mu,\nu)$ satisfies
    $$P_{*}(\mu,\nu)= \id(\mu\wedge\nu) +P_{*}(\mu',\nu')$$ 
  where $\id(\mu\wedge\nu)=(\mu\wedge\nu)\otimes_{x} \delta_{x}$ is the identical coupling of $\mu\wedge\nu$ with itself whereas $\mu'=\mu-\mu\wedge\nu$ and $\nu'=\nu-\mu\wedge\nu$ are the mutually singular parts of~$\mu$ and~$\nu$. \end{proposition}

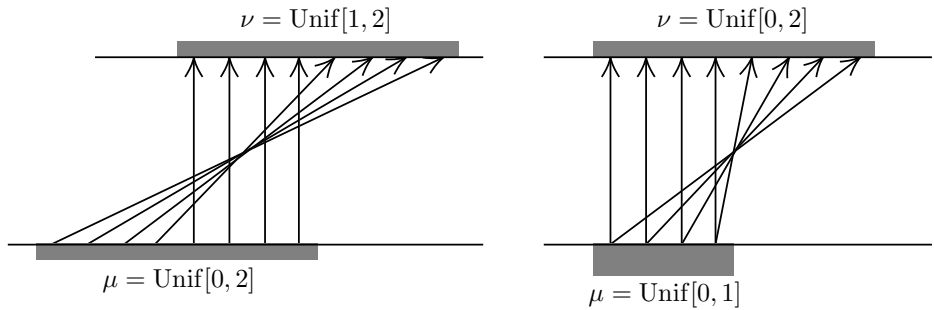
\begin{figure}[b]
\begin{center}
\vspace{-.2em}
\resizebox{\textwidth}{!}{

\tikzset{every picture/.style={line width=0.75pt}} %

\begin{tikzpicture}[x=0.75pt,y=0.75pt,yscale=-1,xscale=1]
\draw [fill={rgb, 255:red, 255; green, 255; blue, 255 }  ,fill opacity=1 ]   (270.7,50.86) -- (50.5,156) ;
\draw [shift={(272.5,50)}, rotate = 154.48] [color={rgb, 255:red, 0; green, 0; blue, 0 }  ][line width=0.75]    (10.93,-4.9) .. controls (6.95,-2.3) and (3.31,-0.67) .. (0,0) .. controls (3.31,0.67) and (6.95,2.3) .. (10.93,4.9)   ;
\draw [fill={rgb, 255:red, 255; green, 255; blue, 255 }  ,fill opacity=1 ]   (249.78,51.01) -- (67.5,158) ;
\draw [shift={(251.5,50)}, rotate = 149.59] [color={rgb, 255:red, 0; green, 0; blue, 0 }  ][line width=0.75]    (10.93,-4.9) .. controls (6.95,-2.3) and (3.31,-0.67) .. (0,0) .. controls (3.31,0.67) and (6.95,2.3) .. (10.93,4.9)   ;
\draw [fill={rgb, 255:red, 255; green, 255; blue, 255 }  ,fill opacity=1 ]   (230.9,51.2) -- (90.5,157) ;
\draw [shift={(232.5,50)}, rotate = 143] [color={rgb, 255:red, 0; green, 0; blue, 0 }  ][line width=0.75]    (10.93,-4.9) .. controls (6.95,-2.3) and (3.31,-0.67) .. (0,0) .. controls (3.31,0.67) and (6.95,2.3) .. (10.93,4.9)   ;
\draw [fill={rgb, 255:red, 255; green, 255; blue, 255 }  ,fill opacity=1 ]   (209.61,51.94) -- (107.5,158) ;
\draw [shift={(211,50.5)}, rotate = 133.91] [color={rgb, 255:red, 0; green, 0; blue, 0 }  ][line width=0.75]    (10.93,-4.9) .. controls (6.95,-2.3) and (3.31,-0.67) .. (0,0) .. controls (3.31,0.67) and (6.95,2.3) .. (10.93,4.9)   ;
\draw [fill={rgb, 255:red, 255; green, 255; blue, 255 }  ,fill opacity=1 ]   (131.97,52.26) -- (131.97,155.26) ;
\draw [shift={(131.97,50.26)}, rotate = 90] [color={rgb, 255:red, 0; green, 0; blue, 0 }  ][line width=0.75]    (10.93,-4.9) .. controls (6.95,-2.3) and (3.31,-0.67) .. (0,0) .. controls (3.31,0.67) and (6.95,2.3) .. (10.93,4.9)   ;
\draw [fill={rgb, 255:red, 255; green, 255; blue, 255 }  ,fill opacity=1 ]   (151.97,52.26) -- (151.97,155.26) ;
\draw [shift={(151.97,50.26)}, rotate = 90] [color={rgb, 255:red, 0; green, 0; blue, 0 }  ][line width=0.75]    (10.93,-4.9) .. controls (6.95,-2.3) and (3.31,-0.67) .. (0,0) .. controls (3.31,0.67) and (6.95,2.3) .. (10.93,4.9)   ;
\draw [fill={rgb, 255:red, 255; green, 255; blue, 255 }  ,fill opacity=1 ]   (171.97,52.26) -- (171.97,155.26) ;
\draw [shift={(171.97,50.26)}, rotate = 90] [color={rgb, 255:red, 0; green, 0; blue, 0 }  ][line width=0.75]    (10.93,-4.9) .. controls (6.95,-2.3) and (3.31,-0.67) .. (0,0) .. controls (3.31,0.67) and (6.95,2.3) .. (10.93,4.9)   ;
\draw [fill={rgb, 255:red, 255; green, 255; blue, 255 }  ,fill opacity=1 ]   (190.97,52.26) -- (190.97,155.26) ;
\draw [shift={(190.97,50.26)}, rotate = 90] [color={rgb, 255:red, 0; green, 0; blue, 0 }  ][line width=0.75]    (10.93,-4.9) .. controls (6.95,-2.3) and (3.31,-0.67) .. (0,0) .. controls (3.31,0.67) and (6.95,2.3) .. (10.93,4.9)   ;
\draw [fill={rgb, 255:red, 255; green, 255; blue, 255 }  ,fill opacity=1 ]   (504.9,51.2) -- (365.97,155.26) ;
\draw [shift={(506.5,50)}, rotate = 143.17] [color={rgb, 255:red, 0; green, 0; blue, 0 }  ][line width=0.75]    (10.93,-4.9) .. controls (6.95,-2.3) and (3.31,-0.67) .. (0,0) .. controls (3.31,0.67) and (6.95,2.3) .. (10.93,4.9)   ;
\draw [fill={rgb, 255:red, 255; green, 255; blue, 255 }  ,fill opacity=1 ]   (484.13,51.45) -- (385.97,155.26) ;
\draw [shift={(485.5,50)}, rotate = 133.4] [color={rgb, 255:red, 0; green, 0; blue, 0 }  ][line width=0.75]    (10.93,-4.9) .. controls (6.95,-2.3) and (3.31,-0.67) .. (0,0) .. controls (3.31,0.67) and (6.95,2.3) .. (10.93,4.9)   ;
\draw [fill={rgb, 255:red, 255; green, 255; blue, 255 }  ,fill opacity=1 ]   (465.5,51.73) -- (405.97,155.26) ;
\draw [shift={(466.5,50)}, rotate = 119.9] [color={rgb, 255:red, 0; green, 0; blue, 0 }  ][line width=0.75]    (10.93,-4.9) .. controls (6.95,-2.3) and (3.31,-0.67) .. (0,0) .. controls (3.31,0.67) and (6.95,2.3) .. (10.93,4.9)   ;
\draw [fill={rgb, 255:red, 255; green, 255; blue, 255 }  ,fill opacity=1 ]   (445.12,51.96) -- (424.97,155.26) ;
\draw [shift={(445.5,50)}, rotate = 101.03] [color={rgb, 255:red, 0; green, 0; blue, 0 }  ][line width=0.75]    (10.93,-4.9) .. controls (6.95,-2.3) and (3.31,-0.67) .. (0,0) .. controls (3.31,0.67) and (6.95,2.3) .. (10.93,4.9)   ;
\draw [fill={rgb, 255:red, 255; green, 255; blue, 255 }  ,fill opacity=1 ]   (365.97,52.26) -- (365.97,155.26) ;
\draw [shift={(365.97,50.26)}, rotate = 90] [color={rgb, 255:red, 0; green, 0; blue, 0 }  ][line width=0.75]    (10.93,-4.9) .. controls (6.95,-2.3) and (3.31,-0.67) .. (0,0) .. controls (3.31,0.67) and (6.95,2.3) .. (10.93,4.9)   ;
\draw [fill={rgb, 255:red, 255; green, 255; blue, 255 }  ,fill opacity=1 ]   (385.97,52.26) -- (385.97,155.26) ;
\draw [shift={(385.97,50.26)}, rotate = 90] [color={rgb, 255:red, 0; green, 0; blue, 0 }  ][line width=0.75]    (10.93,-4.9) .. controls (6.95,-2.3) and (3.31,-0.67) .. (0,0) .. controls (3.31,0.67) and (6.95,2.3) .. (10.93,4.9)   ;
\draw [fill={rgb, 255:red, 255; green, 255; blue, 255 }  ,fill opacity=1 ]   (405.97,52.26) -- (405.97,155.26) ;
\draw [shift={(405.97,50.26)}, rotate = 90] [color={rgb, 255:red, 0; green, 0; blue, 0 }  ][line width=0.75]    (10.93,-4.9) .. controls (6.95,-2.3) and (3.31,-0.67) .. (0,0) .. controls (3.31,0.67) and (6.95,2.3) .. (10.93,4.9)   ;
\draw [fill={rgb, 255:red, 255; green, 255; blue, 255 }  ,fill opacity=1 ]   (424.97,52.26) -- (424.97,155.26) ;
\draw [shift={(424.97,50.26)}, rotate = 90] [color={rgb, 255:red, 0; green, 0; blue, 0 }  ][line width=0.75]    (10.93,-4.9) .. controls (6.95,-2.3) and (3.31,-0.67) .. (0,0) .. controls (3.31,0.67) and (6.95,2.3) .. (10.93,4.9)   ;
\draw  [draw opacity=0][fill={rgb, 255:red, 128; green, 128; blue, 128 }  ,fill opacity=1 ] (356.5,41) -- (514.5,41) -- (514.5,50) -- (356.5,50) -- cycle ;
\draw  [draw opacity=0][fill={rgb, 255:red, 128; green, 128; blue, 128 }  ,fill opacity=1 ] (122.5,41) -- (280.5,41) -- (280.5,50) -- (122.5,50) -- cycle ;
\draw  [draw opacity=0][fill={rgb, 255:red, 128; green, 128; blue, 128 }  ,fill opacity=1 ] (43.5,155) -- (201.5,155) -- (201.5,164) -- (43.5,164) -- cycle ;
\draw  [draw opacity=0][fill={rgb, 255:red, 128; green, 128; blue, 128 }  ,fill opacity=1 ] (356.5,155) -- (435.5,155) -- (435.5,173) -- (356.5,173) -- cycle ;
\draw    (328.5,50.13) -- (546.5,50.13) ;
\draw    (328.5,155.45) -- (546.5,155.45) ;
\draw    (76.5,50.13) -- (294.5,50.13) ;
\draw    (27.5,155.45) -- (294.5,155.45) ;

\draw (396.72,183.61) node    {$\mu =\mathrm{Unif}[ 0,1]$};
\draw (434.72,29.61) node    {$\nu =\mathrm{Unif}[ 0,2]$};
\draw (123.72,175.61) node    {$\mu =\mathrm{Unif}[ 0,2]$};
\draw (200.72,29.61) node    {$\nu =\mathrm{Unif}[ 1,2]$};

\end{tikzpicture}

}
\end{center}
\vspace{-1.9em}\caption{Illustration of Proposition~\ref{pr:decompose} (left) and Example~\ref{ex:randomizedUniform} (right).}\vspace{-1em}
\label{fi:identity}
\end{figure}

A coupling $P$ is of \emph{Monge-type} if $P(Y|X)=T(X)$ is a deterministic function $T$ of $X$ which is then called a Monge map or transport map of~$P$. Equivalently, the stochastic kernel $\kappa$ in the decomposition $P=\mu\otimes\kappa$ has the form $\kappa(x,dy)=\delta_{T(x)}(dy)$ $\mu$-a.s. Proposition~\ref{pr:decompose} suggests that the constrained nature of our transport problem may render $P_{*}$ randomized (i.e., not of Monge-type) even in the absence of atoms.

\begin{example}\label{ex:randomizedUniform}
  Let $\mu=\Unif[0,1]$ and $\nu=\Unif[0,2]$. Then $\mu\lst\nu$ and there are no atoms, yet $P_{*}$ has non-deterministic kernel $\kappa(x)=\frac12(\delta_{x}+\delta_{2-x})$; cf.\ Figure~\ref{fi:identity}. This can be seen, e.g., from the constrained crossing property.
\end{example}

The next results show that this example is representative: the ``coin-flip'' randomization into two maps is the only randomization in $P_{*}$ when $\mu$ is atomless, and it occurs if and only if $\mu\wedge \nu$ and $\mu-\mu\wedge \nu $ are not mutually singular.     The second transport map can also be analyzed in detail. To that end, suppose first that 
$\mu\wedge \nu=0$, so that $(\mu,\nu)$ is already in the reduced form $(\mu',\nu')$ of Proposition~\ref{pr:decompose}. Moreover, suppose for the moment that the marginals are atomless---we discuss later how to reduce atoms to diffuse measures. %
With the convention $\inf\emptyset =\infty$, we have the following (see Figure~\ref{fi:cdfandT} for the graphical interpretation).

\begin{theorem}\label{th:main-transport}
  Let $\mu,\nu$ be atomless and $\mu\wedge \nu=0$. Then $P_{*}$ is of Monge-type with transport map~$T$ given by 
$$%
  T(x)=\inf\{y\geq x:\, (y,F(x))\notin H\}
$$%
for the function $F=F_{\mu}-F_{\nu}$ and its hypograph $H = \{(x,z):\, z\leq F(x)\}$.
\end{theorem}

The proof proceeds by showing that $T$ couples $\mu$ and $\nu$ and that the graph of $T$ satisfies the constrained crossing property. Some of our considerations regarding the local regularity of $F$ may be of independent interest.
Combining the last two results and noting that $F_{\mu}-F_{\nu}=F_{\mu'}-F_{\nu'}$ in Proposition~\ref{pr:decompose}, we deduce the aforementioned assertion on the coin-flip.

\begin{corollary}\label{co:idAndT}
   Let $\mu,\nu$ be atomless. Then
  $$
    P_{*}(\mu,\nu)= (\mu\wedge\nu)\otimes_{x}\delta_{x} + \mu'\otimes_{x}\delta_{T(x)}
  $$
  where 
  $\mu'=\mu-\mu\wedge\nu$. In particular, $P_{*}$ is of Monge-type if and only if $\mu'$ and $\mu\wedge\nu$ are mutually singular.
\end{corollary} 

This result immediately extends to the case where $\nu$ has atoms, essentially by ``filling in'' vertical lines in the graph of $F$ where there are jumps (cf.\ Figure~\ref{fi:cdfandT}). Using a simple transformation detailed in Section~\ref{se:atoms}, it also generalizes to atoms in both marginals, but then $T$ is replaced by a (possibly randomized) coupling; see Theorem~\ref{th:atomsTransform}.

We remark that the invariance property in Corollary~\ref{co:invariance} translates immediately: if $T$ is the map of  $P_*(\mu,\nu)$, then $T^\phi := \phi\circ T\circ \phi^{-1}$ is that of  $P_*(\mu\circ \phi^{-1},\nu\circ \phi^{-1})$; in other words, $T^\phi$ transports $\phi(x)$ to $\phi(y)$ whenever $T$ transports $x$ to $y$.

While we consider the above the main results, three further considerations are presented in Section~\ref{se:furtherProperties}. We discuss when and how $P_{*}$ can be decomposed as a sum of antitone couplings of sub-marginals, remark that $P_{*}$ occurs as optimizer in specific unconstrained transport problems, and finally offer an extension to cone constraints more general  than $Y\geq X$.

\section{Equivalent Characterizations of $P_{*}$}\label{se:characterizations}

In this section we prove Theorems~\ref{th:existenceAndMinimality}--\ref{th:main} and Proposition~\ref{pr:decompose}, the latter being a consequence of the former.
The first step is to show that $\nu_{x}$ in Theorem~\ref{th:existenceAndMinimality} is well-defined. We write $\cM$ for the set of finite measures on $\R$ and recall that  $\theta_1,\theta_2\in\cM$ satisfy $\theta_1\lst \theta_2$ if $\theta_{1}(\R)=\theta_{2}(\R)$ and $F_{\theta_{1}}\geq F_{\theta_{2}}$.

\begin{lemma}\label{le:shadow}
  Let $\mu_{0}\leq \mu$. The set $S=\{\theta\in \cM: \mu_0 \lst \theta  \le \nu \}$ has a unique minimal element $\theta_{*}$; that is, $\theta_{*} \in S$ and  $\theta_{*} \lst \theta$ for all $\theta\in S$. The measure~$\theta_{*}$ has cdf $\sup_{\theta\in S} F_{\theta}$ and we denote $\shadow{\mu_{0}}{\nu} :=\theta_{*}$.
\end{lemma}

\begin{proof}
  We first show that $F:=\sup_{\theta\in S} F_{\theta}$ is a cdf. Given $x<y$, we have $F_\theta (y) -F_\theta (x) \le F_\nu(y) - F_\nu (x)$  for any $\theta \in S$ and hence 
 $$
 F (y) - F (x) \le \sup_{\theta \in S}\, [F_\theta(y)-F_\theta(x)] \le  F_\nu(y) - F_\nu (x) \to 0   \quad\mbox{as}\quad y\downarrow x,
 $$
 showing that $F$ is right-continuous. As the remaining properties of a cdf are immediate, we can introduce $\theta_{*}$ as the measure associated to $F$. In view of $F=\sup_{\theta\in S} F_{\theta}$, we have that $\mu_0\lst \theta_{*}$ and $\theta_{*} \lst \theta$ for every $\theta\in S$. It remains to see that $\theta_{*}  \le \nu$, or equivalently that $F_{\nu-\theta_{*}}$ is nondecreasing. Indeed, 
 $ 
 F_{\nu-\theta_{*}} = F_\nu - \sup_{\theta \in S} F_{\theta}  = \inf_{\theta \in S} F_{\nu-\theta},
 $ 
 and $F_{\nu-\theta}$ is nondecreasing for every $\theta\in S$.
\end{proof}

Next, we show that the map $\mu_{0}\mapsto\shadow{\mu_{0}}{\nu}$ of Lemma~\ref{le:shadow}  is ``divisible'', which is important for its iterated application: mapping $\mu_{0}=\mu_{1}+\mu_{2}$ into $\nu$ produces the same cumulative result as first mapping $\mu_{1}$ and then mapping $\mu_{2}$ into the remaining part of $\nu$.

\begin{lemma}\label{le:shadowAdditive}
  Let $\mu_{1},\mu_{2}$ satisfy $\mu_{1}+\mu_{2}\leq \mu$. Then $\mu-\mu_{1}\lst \nu-\shadow{\mu_{1}}{\nu}$
  and
  $$
    \shadow{\mu_{1}+\mu_{2}}{\nu} = \shadow{\mu_{1}}{\nu} + \shadow{\mu_{2}}{\nu-\shadow{\mu_{1}}{\nu}}.
  $$
\end{lemma}

\begin{proof}
  Let $Q=\mu\otimes \kappa \in \cD(\mu,\nu)$ be arbitrary and let $Q(\mu_{1})$ be its image of~$\mu_{1}$ (that is, the second marginal of $\mu_{1}\otimes\kappa$). In view of $Q\in \cD(\mu,\nu)$ we have $\mu_1\lst Q(\mu_{1})\leq \nu$ and $\mu-\mu_1\lst \nu-Q(\mu_{1})$.
The minimality property of $\shadow{\mu_{1}}{\nu}$ then yields  
$\shadow{\mu_{1}}{\nu} \lst Q(\mu_{1})$ and therefore 
$$\mu_{2}\leq \mu-\mu_{1}\lst \nu-Q(\mu_{1})\lst  \nu-\shadow{\mu_{1}}{\nu}.$$ In particular, the measure $\shadow{\mu_{2}}{\nu-\shadow{\mu_{1}}{\nu}}$ is well defined, and its definition entails $\shadow{\mu_{1}}{\nu} + \shadow{\mu_{2}}{\nu-\shadow{\mu_{1}}{\nu}}\leq \nu$. 
The minimality property of $\shadow{\mu_{1}+\mu_{2}}{\nu}$ now shows that
  \begin{equation}\label{eq:proofshadowAdditive}
    \shadow{\mu_{1}+\mu_{2}}{\nu} \lst \shadow{\mu_{1}}{\nu} + \shadow{\mu_{2}}{\nu-\shadow{\mu_{1}}{\nu}}.
  \end{equation}
  On the other hand, the minimality properties of $\shadow{\mu_{1}}{\nu}$ and $\shadow{\mu_{1}+\mu_{2}}{\nu}$ and direct arguments (omitted for brevity) imply that $\shadow{\mu_{1}}{\nu} \leq \shadow{\mu_{1}+\mu_{2}}{\nu}$. The minimality property of $\shadow{\mu_{1}}{\nu}$ then states in particular that $\shadow{\mu_{1}}{\nu}$ is minimal in stochastic order among all sub-measures of $\shadow{\mu_{1}+\mu_{2}}{\nu}$ with mass $\mu_{1}(\R)$. As a consequence, we see that $$\mu_{2} \lst \shadow{\mu_{1}+\mu_{2}}{\nu} - \shadow{\mu_{1}}{\nu}.$$ Clearly also $\shadow{\mu_{1}+\mu_{2}}{\nu} - \shadow{\mu_{1}}{\nu}\leq \nu - \shadow{\mu_{1}}{\nu}$, and so the minimality property of $\shadow{\mu_{2}}{\nu-\shadow{\mu_{1}}{\nu}}$ implies
  $$
    \shadow{\mu_{2}}{\nu-\shadow{\mu_{1}}{\nu}} \lst \shadow{\mu_{1}+\mu_{2}}{\nu} - \shadow{\mu_{1}}{\nu}.
  $$
  In view of~\eqref{eq:proofshadowAdditive}, the claim follows. 
\end{proof}

We can now construct $P_{*}$.

\begin{proof}[Proof of Theorem \ref{th:existenceAndMinimality}]
  Noting that $\nu-\shadow{\mu|_{(x,\infty)}}{\nu}$ is a nonnegative measure for fixed~$x$, the function
  $$
    F(x,y):=\big(\nu-\shadow{\mu|_{(x,\infty)}}{\nu}\big)(-\infty,y]
  $$
  is clearly nondecreasing and right-continuous in $y$. Moreover, 
  Lemma~\ref{le:shadowAdditive} implies that 
  \begin{equation}\label{eq:checkF1}
   \shadow{\mu|_{(x_{1},\infty)}}{\nu} - \shadow{\mu|_{(x_{2},\infty)}}{\nu} = \shadow{\mu|_{(x_{1},x_{2}]}}{\nu-\shadow{\mu|_{(x_{2},\infty)}}{\nu}}\geq0, \quad x_{1}\leq x_{2}.
  \end{equation}
  The total mass of the right-hand side equals $\mu(x_{1},x_{2}]$ and thus converges to zero as $x_{2}\downarrow x_{1}$, showing that $x\mapsto F(x,y)$ is right-continuous.
  Relation~\eqref{eq:checkF1} also implies that $F$ is supermodular (or nondecreasing on $\R^{2}$): for $x_1\le x_2$ and $y_1\le y_2$,
    \begin{align*}
 [F(x_2,y_2)- F(x_2, & y_1)] - [F(x_1,y_2)-  F(x_1,y_1)]
\\ &  =
   \shadow{\mu|_{(x_{1},\infty)}}{\nu} (y_1,y_2]- \shadow{\mu|_{(x_{2},\infty)}}{\nu} (y_1,y_2]
    \\&  = \shadow{\mu|_{(x_{1},x_{2}]}}{\nu-\shadow{\mu|_{(x_{2},\infty)}}{\nu}}(y_1,y_2] \geq0.
  \end{align*}
  As $F$  has the proper normalization, we conclude (e.g., \cite[p.\,27]{Joe.15}) that $F$ induces a unique probability measure $P_{*}$ on $\cB(\R^{2})$. 
  It remains to observe that $P_{*}\in\cD(\mu,\nu)$. Indeed, the second marginal of~$P_{*}$ is clearly $\nu$. The first marginal is equal to $\mu$ as   for each $x$, $$\lim_{y\to\infty} F(x,y)=\nu(\R)-\theta^{\nu}(\mu|_{(x,\infty)})(\R)=1-\mu((x,\infty))=\mu((-\infty,x]).$$ Finally, $P_{*}$ is directional since
  $$
    P_{*}\big((x,\infty)\times (-\infty,x]\big) = \shadow{\mu|_{(x,\infty)}}{\nu}(-\infty,x]=0,\quad x\in\R
  $$
  due to the fact that $\mu|_{(x,\infty)}\lst\shadow{\mu|_{(x,\infty)}}{\nu}$ by the definition of $\shadow{\cdot}{\nu}$.
\end{proof}

\begin{remark}\label{rk:minimalOpen}
  While we have defined $P_{*}$ as mapping $\mu|_{(x,\infty)}$ to $\shadow{\mu|_{(x,\infty)}}{\nu}$, it equivalently maps $\mu|_{[x,\infty)}$ to $\shadow{\mu|_{[x,\infty)}}{\nu}$ for all $x\in\R$. This follows from Theorem~\ref{th:existenceAndMinimality} and Lemma~\ref{le:shadowAdditive}.
\end{remark}

We now turn the the equivalent characterizations in Theorem~\ref{th:main}; here the most important tool is the notion of cyclical monotonicity in optimal transport (e.g., \cite{GangboMcCann.96, Villani.09}).

\begin{proof}[Proof of Theorem \ref{th:main}]
   Given two probability measures $P,Q$ on $\R^{2}$ with the same marginals, it is known that the  concordance order $F_{P}\leq F_{Q}$ is equivalent to $\int g   \,dP \geq \int g   \,dQ$ for all (suitably integrable) supermodular $g$; cf.\ \cite[Theorem~3.8.2, p.\,108]{MullerStoyan.02}. The implication (i)$\Rightarrow$(ii)  is a direct consequence of that fact, and (ii)$\Rightarrow$(iii) is trivial. 
  
  (iii)$\Rightarrow$(iv): Let $g$ be Borel and $(\mu,\nu)$-integrable. We consider the (unconstrained) Monge--Kantorovich optimal transport problem on $\R\times\R$ with marginals $(\mu,\nu)$ and cost function
  $$
  c(x,y)=
  \begin{cases}
  -g(x,y), & (x,y)\in \H,\\
  \infty, & \mbox{otherwise.}
  \end{cases} 
  $$
  Noting that $c(x,y)\geq \phi(x)+\psi(y)$ for some $\phi\in L^{1}(\mu)$ and $\psi\in L^{1}(\nu)$, it follows from \cite[Theorem~1(a)]{BeiglbockGoldsternMareschSchachermayer.09} that any optimal transport $P$ is concentrated on a Borel set $\Gamma\subseteq \R^{2}$ that is $c$-cyclically monotone.
   As no transport with finite cost charges the complement $\H^{c}$, we may replace $\Gamma$ with $\Gamma\cap \H$ to ensure that $\Gamma\subseteq \H$. Cyclical monotonicity then states in particular that\footnote{More generally, the monotonicity holds for cycles of finite length $n$; that is, $\sum_{i=1}^{n} g(x_{i},y_{i}) \geq \sum_{i=1}^{n} g(x_{i},y_{\pi(i)})$ for all $(x_{1},y_{1}),\dots, (x_{n},y_{n}) \in \Gamma$ and permutations $\pi$ of $\{1,\dots,n\}$. The stated property corresponds to $n=2$.}
  $$
    g(x,y) + g(x',y') \geq g(x,y') + g(x',y)\quad\mbox{for all}\quad (x,y),(x',y')\in \Gamma.
  $$
  Thus, if $g$ is strictly submodular, $\Gamma$ cannot contain improvable pairs.

  (iv)$\Rightarrow$(v): Suppose for contradiction that $P\ne P_{*}$. In view of Theorem~\ref{th:existenceAndMinimality}, there exists $x\in \R$ such that $P$ maps $\mu |_{(x,\infty)}$ to a measure $\nu_x'\neq\nu_{x}$, and $\nu_{x} \lst \nu_{x}'$ by the minimality property of $\nu_x$. It follows from Lemma~\ref{le:new1} below that there exist  $z>y\geq x$ such that
  \begin{equation}\label{eq:proofCyclOpt}
    \nu_x( (x,y] )> \nu_x' ( (x,z)) \quad\mbox{and}\quad  \nu_x ([y,z)) > \nu_x'([y,z)).
  \end{equation}
  Using also that $\mu((x,y])\geq \nu_{x}((x,y])$ due to $\mu |_{(x,\infty)}\lst \nu_{x}$, we deduce  \begin{align*}
  P((x, y] \times [z, \infty))  
    \geq  \mu((x,y]) - \nu_x' ((x, z))
    \geq \nu_x( (x,y] ) - \nu_x' ((x, z))  > 0.
\end{align*}
    By the constrained crossing property, this implies $ P((-\infty,x]) \times [y,z) )  = 0$ and thus
  \begin{align*} 
    \nu(  [y,z)  ) &=  P( \R \times [y,z) )  
      = P((x,\infty) \times [y,z) ) =\nu_x' ([y, z)),
    \end{align*}
  contradicting~\eqref{eq:proofCyclOpt}.

 (v)$\Rightarrow$(i): Let $x,y\in \R$; we show $F_{P_{*}}(x,y)\le F_{Q}(x,y)$ for $Q\in \cD(\mu,\nu)$. As $P_{*}$ and $Q$ have the same second marginal, this is equivalent to
  \begin{equation*}%
    P_{*}((x,\infty)\times (-\infty,y])  \geq Q  ((x,\infty)\times (-\infty,y]).
  \end{equation*}  
  Recalling $\nu_{x}$ from Theorem \ref{th:existenceAndMinimality} and denoting by $\theta$ the measure that $\mu |_{(x,\infty)}$ is transported to by $Q$, the above can be stated as $\nu_{x}((-\infty,y])\geq \theta((-\infty,y])$, and that clearly follows from the formula for $F_{\nu_{x}}$ in Theorem \ref{th:existenceAndMinimality}.
\end{proof}

The following was used in the preceding proof of (iv)$\Rightarrow$(v).

 \begin{lemma}\label{le:new1}
 Given $\mu_1,\mu_2\in \cM$ with $\mu_1\lst \mu_2$ and $\mu_1\ne \mu_2$,
 there exist $z>y $ such that
\begin{equation} \label{eq:existence} \mu_1( (-\infty,y] )> \mu_2 ( (-\infty,z)) 
 \mbox{~~~and~~~} 
  \mu_1 ([y ,z)) > \mu_2 ([y ,z)) .\end{equation}
 \end{lemma}
  
\begin{proof}
  Define two real functions
  $$
    \phi^+(y) = \mu_1( (-\infty,y] ) -\mu_2 ( (-\infty,y]), ~~ \phi^-(y) = \mu_1( (-\infty,y) ) -\mu_2 ((-\infty,y)).
  $$
  Then $\phi^+ $ and $\phi^{-}$ are right- and left-continuous,  respectively, both are nonnegative, and $\phi^+(y)=\phi^-(y)$ whenever $\mu_1(\{y\})=\mu_2(\{y\})$. If $y\in \R$ satisfies 
   \begin{equation}\label{eq:yz1} 
  \phi^+(y)>0\mbox{ ~~~ and ~~ }
  \sup_{z\in (y,y+\eps)} \phi^-(z) >\phi^-(y) \mbox{~for each $\eps>0$,}
  \end{equation}
   then \eqref{eq:existence} holds by choosing $z>y$ close enough to $y$. We argue by contradiction and suppose that there is no $y\in \R$ satisfying~\eqref{eq:yz1}.
   Thus, if $\phi^+(y)>0$, there exists $\eps>0$ such that $\phi^{-}(z)\le \phi^{-}(y)$ for $z\in (y,y+\eps)$. This implies that the function $\phi^{-}$ has no upward jumps; i.e.,  $\Delta \phi^{-}\leq0$.  As $\mu_1\ne\mu_2$, there exists $y_0\in \R$ such that $\phi^{-}(y_0)>0$. Since $\phi^{-}(y)\to 0$ as $y\downarrow -\infty$ and there are no upward jumps, 
  there exists $y_1<y_{0}$ such that $0<\phi^{- }(y_1)<\phi^{-}(y_0)$. 
   Let $y = \inf\{z>y_1: \phi^{-}(z)>\phi^{-}(y_1) \} $. 
  Then the left-continuity of $\phi^{-}$
  implies $y <y_{0}$ and the absence of upward jumps implies 
  $\phi^{-}(y)=\phi^{-}(y_1)$ as well as that $y$ cannot be the location of a downward jump. Therefore, $\mu_1(\{y\})=\mu_2(\{y\})$ and $\phi^{+}(y)=\phi^{-}(y)=\phi^{-}(y_1)>0$. Finally, given $\eps>0$, we have $\phi^-(z) > \phi^{-}(y_1) = \phi^- (y)$ for some $z\in (y,y+\eps)$ by the definition of~$y$, so that $y$ satisfies~\eqref{eq:yz1} and we have reached a contradiction.
\end{proof}

\begin{remark}\label{rk:integrability}
  The integrability condition in Theorem~\ref{th:main} can we weakened to the positive part $g^{+}$ being $(\mu,\nu)$-integrable and the negative part satisfying $\int g^{-} \, dP<\infty$ for some $P\in\cD$, so that the value function is not trivial. %
\end{remark}

The final task of this section is to deduce the decomposition in Proposition~\ref{pr:decompose}
from Theorem~\ref{th:main}.

\begin{proof}[Proof of Proposition~\ref{pr:decompose}] 
  By Theorem~\ref{th:main}, the optimal coupling $P_{*}(\mu',\nu')$ of $\mu',\nu'$ is supported by a set $\Gamma'$ with the constrained crossing property. Define $\Gamma=\Gamma'\cup \Delta$ where $\Delta=\{(x,x):\,x\in\R\}$ is the diagonal in $\R^{2}$, then $\Gamma$ again has the constrained crossing property. Set $P=\id(\mu\wedge\nu) +P_{*}(\mu',\nu')$ and note  $P\in\cD(\mu,\nu)$. As $\Delta$ supports the identical coupling, $P$ is supported by $\Gamma$ and (iv)$\Rightarrow$(v) of Theorem~\ref{th:main} shows that $P=P_{*}(\mu,\nu)$.
\end{proof}

\section{Joint Distribution Function}\label{se:cdf}

As mentioned in Section~\ref{se:mainResults}, the formula for the joint distribution function~$F_*$ of~$P_*$ in Corollary~\ref{co:cdf} can be deduced from Theorem~\ref{th:main}\,(i) and \cite[Theorem~6]{ArnoldMolchanovZiegel.20} which uses arguments from copula theory. Below, we sketch a direct derivation and some consequences.

\begin{proof}[Proof of Corollary~\ref{co:cdf}]
As $P_{*}$ is directional, $y \le x$ implies 
$$
 F_{*}(x,y) =P_{*}((-\infty,x]\times (-\infty,y]) =\nu(  (-\infty,y]) = F_{\nu}(y),
$$
so we can focus on $y>x$. Denote $c=\inf_{z\in [x,y]} (F_{\mu}(z)-F_{\nu}(z))$ and recall that $X,Y$ are the coordinate projections. We first consider an arbitrary $P\in \cD(\mu,\nu)$. Then as $X\le Y$ $P$-a.s., we have for $z\in [x,y]$ that
\begin{align*}
P(X\le z,X>x)& \ge P(Y\le z, X>x)
\\&=P(Y\le z) -P(Y\le z, X\le x) 
\\&= P(Y\le z) -P(X\le x) +P(Y> z, X\le x) 
\\&\ge P(Y\le z) -P(X\le x) +P(Y> y, X\le x);
\end{align*} 
that is,  
$
F_{\mu}(z)-F_{\mu}(x) \ge F_{\nu}(z) -F_{\mu}(x) +P(Y> y, X\le x). 
$
This shows 
$P(Y> y, X\le x) \le \inf_{z\in [x,y]}(F_{\mu}(z)-F_{\nu}(z))=c$
and we conclude that
\begin{align}
 F_{P}(x,y)
&=P(X\le x) -
P(Y>y,X\le x) 
\ge F_{\mu}(x) -c. \label{eq:cdf-eq2}
\end{align}
In view of Theorem~\ref{th:main} we have $F_{*}(x,y)=\inf_{P\in\cD(\mu,\nu)} F_{P}(x,y)$. Thus, to complete the proof, it suffices to show that some $P\in\cD(\mu,\nu)$ attains equality in the above inequality.

Let $a=F_{\mu}(x)$ and $b=F_{\nu}(y)$; note that
 $0\leq c\le a\le b+c \le 1.$  
Let $U \sim\Unif [0,1]$ and define  a random variable $V$ as 
$$
V=\begin{cases}
U+b+c -a, &  a-c < U \le a,\\
U-c, & a< U \le b+c,
\\
U,& \mbox{otherwise}.
\end{cases}
$$
Then $V\sim\Unif [0,1]$ like $U$, and thus
$
  P := \Law (F^{-1}_\mu (U), F^{-1}_{\nu}(V))
$
has marginals $\mu$ and $\nu$, respectively. One checks by direct arguments that~$P$ is directional.
Finally, if $U\in (a-c,a]$, 
then $F^{-1}_\mu (U) \le F^{-1}_\mu(a) \le x$ 
and $F^{-1}_{\nu}(V) \ge F^{-1}_{\nu}(b+)$, so that
$P(X\le x, Y>y) \ge P(U\in (a-c,a]) =c$. This shows that $P$ attains equality in~\eqref{eq:cdf-eq2}.
\end{proof}

\begin{remark}\label{rk: }
  One can give a yet another proof of Corollary~\ref{co:cdf} based on Theorem~\ref{th:main-transport} below, as may be intuitive given Figure~\ref{fi:cdfandT}. %
\end{remark}

Corollary~\ref{co:cdf} implies that $P_*$ is continuous with respect to the marginals as stated in Corollary~\ref{co:continuityWrtMarginals}. The next example shows that this assertion may fail if the limiting marginals have atoms, a phenomenon caused by the directional constraint.

\begin{example}[Discontinuity wrt.\ Marginals] \label{ex:non-convergence}
For $n\in \N$, let $\mu_n$ and $\nu_n$ be such that $\mu_n\{0\}=\mu_n\{1\} =1/2$ and $\nu_n\{1-1/n\}=\nu_n\{2\} =1/2$. 
Then $\mu_n\lst \nu_n$ and $\nu_n\weakto \nu$ with $\nu\{1\}=\nu\{2\}=1/2$, and  $\mu_{n}\equiv \mu$ is constant.
We see that 
$P_*(\mu_n,\nu_n)$ is the comonotone coupling, 
$P_*(\mu,\nu)$ is the antitone coupling, and $P_*(\mu_n,\nu_n)  \notweakto P_*(\mu,\nu)$.
\end{example}

Another consequence are simple bounds on $F_{*}$. A right-continuous function on~$\R$ is \emph{unimodal} if it is nondecreasing on $(-\infty,x_{0})$ and nonincreasing on $[x_{0},\infty)$ for some $x_{0}\in\R$.

\begin{corollary}\label{coro:bounds} 
We have $H^{\wedge} \le F_{*} \le H^{\vee}$ for
\begin{align*}
 H^{\wedge} (x,y) &= F_\nu(y) - [(F_\mu(y)-F_{\mu}(x))\wedge (F_\nu(y)-F_{\nu}(x))]_+,\\
 H^{\vee}(x,y)&=F_{\mu}(x)\wedge F_{\nu}(y).
\end{align*}
 \begin{enumerate} 
\item  $F_{*}=H^{\wedge}$ if and only if $F=F_\mu-F_{\nu}$ is unimodal. 
\item $F_{*} = H^{\vee}$ if and only if $\cD(\mu,\nu)$ is a singleton. If, in addition, $F$ is continuous, these conditions are further equivalent to $\mu=\nu$.
\end{enumerate}  
\end{corollary}

\begin{proof}
The lower bound follows by considering $z\in\{x,y\}$ in~\eqref{eq:cdf}. The upper bound follows directly from~\eqref{eq:cdf}; alternately, it can also be obtained by noting that $H^{\vee}$ is the cdf of the comonotone coupling.
 
To see (i), note that by~\eqref{eq:cdf}, $F_{*}=H^{\wedge} $ if and only if 
$\min_{z\in[x,y]}  F(z) = F(x)\wedge F(y)$ for all $x<y$.
This is equivalent to $F$ being unimodal.
Turning to (ii), we first recall from Theorem~\ref{th:main}\,(i) that $P_{*}$ has the minimal cdf in $\cD(\mu,\nu)$. On the other hand, $H^{\vee}$ is the cdf of the comonotone coupling, which is the maximal cdf among all couplings and in particular in $\cD(\mu,\nu)$. Thus, $F_{*}=H^{\vee}$ if and only if all directional couplings have the same cdf, showing the first claim.
Now let $F$ be continuous and suppose for contradiction that 
$\mu\ne \nu$. In view of Proposition~\ref{pr:decompose}, we may assume that $\mu\wedge\nu=0$.
By Lemma~\ref{pr:HahnDecomposition}, $\mu(I)>0$ for the set $I$ of strict increase of $F$. In particular, there exists $x\in I$, which implies that $F_{\mu}(x)> F_{\nu}(x)$ and 
$\mu((x,z])>0$ for any $z>x$.  As $P_{*}$ is the comonotone coupling, $\mu|_{(x,\infty)}$ is transported to $\nu|_{(y,\infty)}$ for some $y>x$. On the other hand, $\nu((x,y])>0$ due to $\mu((x,\infty)) =\nu((y,\infty))< \nu ((x,\infty))$, which by minimality implies that $\nu_{x}$ charges $(x,y]$, contradicting $\nu_{x}=\nu|_{(y,\infty)}$. Conversely, $\mu=\nu$ clearly implies that the identity is the only directional coupling.
\end{proof}
 
\begin{remark}\label{rk:earlierResults}
(a) In view of Theorem~\ref{th:main}\,(i), the lower bound $F_{*} \ge H^{\wedge}$ is equivalent to the statement that $F_{Q} \ge H^{\wedge}$ for all $Q\in\cD(\mu,\nu)$. The latter result was first obtained in \cite{Smith.83}. See also~\cite{SarkarSmith.86} for a lower bound on a different coupling in a similar spirit. Both upper and lower bound were noted in \cite{ArnoldMolchanovZiegel.20}, where it was also observed that the lower bound holds in the case of unimodality. The sharpness conditions are novel, to the best of our knowledge.

(b) The continuity assumption in~(ii) is clearly important for the last conclusion: if $\mu$ is a Dirac mass, all couplings of $\mu$ and $\nu$ coincide and in particular $F_{*} = H^{\vee}$, but of course $\mu$ and $\nu$ need not be equal.
\end{remark}

The following is a standard example satisfying the condition in Corollary~\ref{coro:bounds}\,(i) and covering, for instance, two normal or exponential marginals in stochastic order. The appearance of an antitone coupling is a particular case of a phenomenon that will be discussed in detail in Section~\ref{se:antitone}.

\begin{example}[Single-crossing Densities]
\label{ex:single-crossing}
Suppose that 
 $\mu$ and $\nu$ have densities $f_\mu$ and $f_\nu$ which cross exactly once; that is, there exists a point $x_0\in \R$ such that 
$f_\mu(x) \ge f_\nu(x)$ for $x\le x_0$ and 
$f_\mu(x) \le f_\nu(x)$ for $x\ge x_0$. 
Then $F$ is unimodal and hence $F_{*}=H^{\wedge}$.  
By Proposition~\ref{pr:decompose}
and the fact that the measures $\mu'$ and $\nu'$ (defined therein) are supported on disjoint sets, we see that 
 $P_*(\mu,\nu)$ is the sum of an identity coupling  $\id(\mu\wedge \nu)$ 
and an antitone coupling   $P_*(\mu',\nu')$. 
\end{example}

\section{The Transport Map}\label{se:transportMap}

The aim of this section is to prove Theorem~\ref{th:main-transport} on the optimal transport map~$T$. The analysis rests on a specific Hahn decomposition that holds for arbitrary signed, diffuse measures on $\R$ and is provided in the first subsection. We then return to our transport problem, showing in Sections~\ref{se:Tbasics}--\ref{se:TcouplesAndCrossing} that~$T$ induces a coupling with the constrained crossing property, and thus is optimal. Section~\ref{se:atoms} explains how marginals with atoms can be reduced to the continuous case by a simple transformation.

\subsection{Sets of Increase and Decrease}
\label{se:tech1}

Let $F: \R\to\R$ be a continuous function of bounded variation. We recall that the signed measure $\rho$ associated  to $F$ admits a unique Jordan decomposition $\rho=\mu-\nu$ into mutually singular nonnegative measures, and then $\tau=\mu+\nu$ is the total variation measure of $\rho$.
(In this section, $\mu$ and $\nu$ are arbitrary finite measures---not necessarily of the same mass or even $\mu \lst \nu$.) 
 Similarly to $\rho$, the function $F$ can be uniquely decomposed as $F=F_{\mu}-F_{\nu}$ into continuous nondecreasing functions that are mutually singular; that is, $V:=F_{\mu}+F_{\nu}$ is the total variation of $F$. Disjoint Borel sets $B_{\mu}, B_{\nu}$ form a \emph{Hahn decomposition} for $\rho$  (or $F$)
if $\mu(B_{\mu}^{c})=\nu(B_{\nu}^{c})=0$ and $\mu(B_{\nu})=\nu(B_{\mu})=0$. In particular, $\tau$ is then carried by $B_{\mu}\cup B_{\nu}$.

If $F$ is of class $C^{1}$, the sets $\{\partial F>0\}$ and $\{\partial F<0\}$ clearly form a Hahn decomposition. Moreover, the two sets are countable unions of intervals where $F$ is monotone. Our purpose is to provide a similar Hahn decomposition for bounded variation functions---here the sets will merely be Borel, as it is well known that a function can be absolutely continuous without being monotone on any interval (e.g., \cite[p.\,109, Exercise~41]{Folland.99}).

Consider a function $F: \R\to\R$ and $x\in\R$. We call $x$ a \emph{point of strict increase} if there is a neighborhood of $x$ in which $x_{0}<x<x_{1}$ implies $F(x_{0})<F(x)<F(x_{1})$. The set of all such points is called the \emph{set of strict increase} of $F$ and denoted $I_{F}$. Points of strict decrease are defined analogously, and their set is denoted $D_{F}$.

\begin{proposition}\label{pr:HahnDecomposition}
  Let $F: \R\to\R$ be a continuous function of bounded variation. The sets $I_{F},D_{F}$ of strict increase and decrease form a Hahn decomposition for $F$.
\end{proposition} 

\begin{proof}
  \emph{Step 1.} Let $\mu,\nu,\tau$ and $F_{\mu},F_{\nu},V$ be as introduced above. Clearly $\mu,\nu$ admit densities $f_{\mu},f_{\nu}$ with respect to $\tau$, and these can be chosen to be indicator functions of complementary sets by the Hahn decomposition theorem. That is, $f_{\mu}(x),f_{\nu}(x)\in\{0,1\}$ and $f_{\mu}(x)+f_{\nu}(x)=1$ for all $x\in\R$.

Next, we claim that (with $z/0:=0$, say) the limit 
  $$
    f(x):=\lim_{\eps\to0 } \frac{\mu([x,x+\eps])}{\tau([x,x+\eps])}
  $$
  exists for $\tau$-a.e.\ $x\in\R$ and defines a version of the Radon--Nikodym derivative $d\mu/d\tau$---existence meaning particular that the limit is the same along any sequence  $0\neq\eps_{n}\to 0$. 
  Let $V^{-1}$ be the right-continuous inverse of $V$. Then $F_{\mu}\circ V^{-1}$ is nondecreasing and $F_{\mu}\ll V$ implies that $\mu_{F_{\mu}\circ V^{-1}}\ll\lambda$, where $\mu_{F_{\mu}\circ V^{-1}}$ is the Lebesgue--Stieltjes measure of $F_{\mu}\circ V^{-1}$ and $\lambda$ is the Lebesgue measure. By Lebesgue's differentiation theorem
 \cite[Theorem 3.21, p.\,98]{Folland.99},  
  $F_{\mu}\circ V^{-1}$ is $\lambda$-a.e.\ differentiable and the derivative $\phi$ defines a density $d\mu_{F_{\mu}\circ V^{-1}}/d\lambda$. (In fact, $F_{\mu}\circ V^{-1}$ is even Lipschitz.) That is, there exists a Lebesgue-nullset $N_{\lambda}$ such that for $y\notin N_{\lambda}$ and $y'\to y$,
  $$
    \phi(y)=\lim_{y'\to y} \frac{F_{\mu}(V^{-1}(y')) - F_{\mu}(V^{-1}(y))}{y' - y} 
  $$
  exists. 
  Let $N=V^{-1}(N_{\lambda})$; then $\tau(N)=0$ as $\tau=\lambda\circ V$. For $x\notin N$ we have $y:=V(x)\notin N_{\lambda}$. As $V$ is continuous and $F_{\mu}=F_{\mu}\circ V^{-1}\circ V$, using the above with $y'=V(x')$ yields that
  $$
    f(x) = \lim_{x'\to x} \frac{F_{\mu}(x') - F_{\mu}(x)}{V(x') - V(x)} = \lim_{y'\to y} \frac{F_{\mu}(V^{-1}(y')) - F_{\mu}(V^{-1}(y))}{y' - y} 
  $$  
  exists and satisfies $f(x)=\phi(V(x))$. 
    By the change-of-variable formula we see that  $f$ is a density of $\mu$ with respect to $\tau$.
  It now follows that $f=f_{\mu}$ $\tau$-a.e. As a result, for all $x$ outside a $\tau$-nullset and any sequence $\eps_{n}\to0$,
  $$
    f(x)=\lim_{n}\frac{F_{\mu}(x+\eps_{n})-F_{\mu}(x)}{F_{\mu}(x+\eps_{n})-F_{\mu}(x) + F_{\nu}(x+\eps_{n})-F_{\nu}(x)} \in\{0,1\}.
  $$  

  \emph{Step 2.} Let $I=I_{F}$, $D=D_{F}$. The set $(I\cup D)^{c}$ consists of three types of points. First, the  strict local minimum and maximum points; this subset is countable and hence a $\tau$-nullset as $V$ is continuous. Second, the points which are contained in an interval of constancy of $F$. There are countably many such intervals and each one is clearly a $\tau$-nullset. Third, the points of oscillation: If $x\in (I\cup D)^{c}$ is not in an interval of constancy of $F$ and if $0\neq \eps_{n}\to0$, then for all $n$ large we have either $F_{\mu}(x+\eps_{n})-F_{\mu}(x)\neq0$ or $F_{\nu}(x+\eps_{n})-F_{\nu}(x)\neq0$. If, in addition, $x$ is not a strict local extremum, continuity implies that there exist $0\neq \eps_{n}\to0$ such that $F(x)=F(x+\eps_{n})$; that is, $F_{\mu}(x+\eps_{n})-F_{\mu}(x)= F_{\nu}(x+\eps_{n})-F_{\nu}(x)$. Combining these two properties,
  $$
    F_{\mu}(x+\eps_{n})-F_{\mu}(x)= F_{\nu}(x+\eps_{n})-F_{\nu}(x) \neq 0
  $$
  for all $n$ large. In particular,
  $$
    \frac{F_{\mu}(x+\eps_{n})-F_{\mu}(x)}{F_{\mu}(x+\eps_{n})-F_{\mu}(x) + F_{\nu}(x+\eps_{n})-F_{\nu}(x)}\to \frac12.
  $$
  In view of Step~1, the set of all such $x$ must be a $\tau$-nullset. This completes the proof that $(I\cup D)^{c}$ is $\tau$-null. It is easy to see that $I$ and $D$ are disjoint Borel sets. Noting also that $\{f=1\}\subseteq I$ and $\{f=0\}\subseteq D$, it follows that $I,D$ form a Hahn decomposition. 
\end{proof}

\subsection{Basic Properties of $T$}\label{se:Tbasics}

We return to our setting with given marginals $\mu\lst\nu$. Throughout this section we assume that $\mu \wedge \nu=0$, or equivalently, that $\mu$ and $\nu$ are mutually singular. For simplicity of exposition, we first focus on the case of diffuse marginals $\mu$ and $\nu$; the extension to measures with atoms is then simple and carried out in Section~\ref{se:atoms}.

We consider $F=F_{\mu}-F_{\nu}$, a nonnegative continuous function of bounded variation with $F(-\infty)=F(\infty)=0$,  its graph $G$ and its hypograph $H$,
$$
  G = \{(x,z):\, z= F(x)\}, \qquad H = \{(x,z):\, z\leq F(x)\}.
$$
Recall from Theorem~\ref{th:main-transport} that
\begin{equation}\label{eq:def:t}
  T(x)=\inf\{y\geq x:\, (y,F(x))\notin H\}
\end{equation}
for $x\in \R$, with the convention $\inf\emptyset=\infty$. 
Let $I=I_{F}$ and $D=D_{F}$ be the sets of strict increase and decrease of~$F$, respectively (see Section~\ref{se:tech1}).

\begin{lemma}\label{le:Tpositive}
  We have $\mu(I)%
  =\nu(D)=1$. The function $T$ is upper semicontinuous and bimeasurable, it satisfies $(T(x),F(x))\in G$ whenever $T(x)<\infty$, and $T(x)=\infty$ if and only if $F(x)=0$.
\end{lemma} 

\begin{proof}
The  statement $\mu(I)= \nu(D)=1$ follows directly from Proposition~\ref{pr:HahnDecomposition} since $I,D$ form a Hahn decomposition for $F$ and $\mu\wedge\nu=0$. As $H$ is closed, $T$ is upper semicontinuous. In view of $G=\partial H$, we also have $(T(x),F(x))\in G$ whenever $T(x)<\infty$. Finally, $F(\infty)=0$ implies that $T(x)=\infty$ if and only if $F(x)=0$. To see that $T$ is bimeasurable---i.e., also satisfies $T(\cB(\R))\subseteq \cB(\R)$---it suffices to show that there are at most countably many points~$y$ whose preimage $T^{-1}(y)$ is uncountable; see for instance \cite[Main~Theorem]{Purves.66}. 
  Let $y$ be such that $T^{-1}(y)$ contains more than one point. The construction of $T$ shows that all elements $x\in T^{-1}(y)$, except possibly one, are local minima of $F$, and they have the common value $F(x)=F(T^{-1}(y))$. Any real function $f$ only has countably many local minimum values~$f(x)$ (because each local minimum is minimal within a rational interval, yielding an injection of the minimum values into $\Q^{2}$), so it suffices to show that for fixed $y$, $T^{-1}(y)$ contains at most countably many points $x$ which also have the property that $T^{-1}(x)$ has several elements. If $x_{0}<x$ is such that $T(x_{0})=x$, it follows that $T^{-1}(x')=\emptyset$ for all $x'\in (x_{0},x)$ with $F(x')=F(x)$. Thus we can associate with $x$ an interval of positive length in which it is unique with the property in question, and that implies the claim.
\end{proof}

\subsection{Marginals and Geometry of $T$}\label{se:TcouplesAndCrossing}

\begin{lemma}\label{le:Tcouples}
  The map $T$ transports $\mu$ to $\nu$.
\end{lemma}

\begin{proof}
We show that $\mu \{T \le y\} =\nu ((-\infty,y])$ for $y\in \R$.
Define the continuous function
$$
  M(x) = F(x) -  \min_{z\in [x,y]}F(z) \ge 0, \quad x\in (-\infty, y].
$$
For $x\in I$ with $x\le y$, $M(x)>0$ is equivalent to the existence of $z\in (x,y]$ such that $F(z)<F(x)$, thus equivalent to $T(x)\le y$. As $\mu$ is concentrated on $I$ and $T$ is directional, it follows that
 $$
   \mu\{T \le y\} = \mu \{x\in (-\infty,y]: T(x)\le y\} = \mu( (-\infty,y]\cap \{M>0\}).
 $$
On the other hand, $M>0$ on $D\cap(-\infty,y)$ and $\nu$ is concentrated on $D$, hence $\nu ((-\infty,y]) = \nu( (-\infty,y]\cap \{M>0\})$ and it suffices to show that
$$
  (\mu-\nu)( (-\infty,y]\cap \{M>0\})=0.
$$
Noting that $M(-\infty)=M(y)=0$, we see that the set $(-\infty,y]\cap \{M>0\}$ is open and thus is the union of countably many open intervals
of the form $J=(a,b)$ with $M(a)=M(b)=0$ and $M>0$ on $J$. The last two facts and the definition of $M$ imply that
$$
  F(a)= \min_{z\in [a,y]}F(z) = \min_{z\in [b,y]} F(z) = F(b)
$$
and hence $(\mu-\nu)(J)=F(b)-F(a)=0$, completing the proof.
\end{proof}

\begin{lemma}\label{le:TconstrainedCrossing}
  If $x',x\in\R$ satisfy $x'<x\leq T(x')$, then $T(x')\geq T(x)$. In particular, the graph of $T$ has the constrained crossing property.
\end{lemma} 

\begin{proof}
  Let $x',x\in\R$ satisfy $x'<x\leq T(x')$. Note that $F(x)<F(x')$ would imply $(x,F(x'))\notin H$ and hence $T(x')<x$, a contradiction. Thus,  $F(x')\leq F(x)$. The semi-infinite rectangle
  $
    R=\{(a,b): x\leq a\leq T(x),\; b\leq F(x)\}
  $
  is contained in the hypograph $H$, and similarly for the rectangle $R'$ defined with $x'$ instead of $x$ (cf.\ Figure~\ref{fi:TconstrainedCrossing}).
\begin{figure}[thb]
\begin{center}
\resizebox{7cm}{!}{
\tikzset{every picture/.style={line width=0.75pt}} %
\begin{tikzpicture}[x=0.75pt,y=0.75pt,yscale=-1,xscale=1]
\draw  (53,167.8) -- (355.5,167.8)(83.25,31) -- (83.25,183) (348.5,162.8) -- (355.5,167.8) -- (348.5,172.8) (78.25,38) -- (83.25,31) -- (88.25,38)  ;
\draw   (316.82,122.01) .. controls (243.94,29.3) and (171.93,29.98) .. (100.81,124.04) ;
\draw  [dash pattern={on 4.5pt off 4.5pt}]  (151.5,73) -- (162.5,73) -- (267.5,73) ;
\draw    (125.5,96) -- (293.5,96) ;
\draw    (125.5,96) -- (125.5,177) ;
\draw    (293.5,96) -- (293.5,177) ;
\draw  [dash pattern={on 4.5pt off 4.5pt}]  (151.5,73) -- (151.5,177) ;
\draw  [dash pattern={on 4.5pt off 4.5pt}]  (267.5,73) -- (267.5,177) ;
\draw (153.5,183) node    {$x$};
\draw (127,181) node    {$x'$};
\draw (301.5,187) node    {$T( x')$};
\draw (264.5,187) node    {$T( x)$};
\draw (163.5,83) node    {$R$};
\draw (138.5,106) node    {$R'$};
\draw (325.5,123) node    {$F$};
\end{tikzpicture}
}
\end{center}
\vspace{-1em}\caption{On the proof of Lemma~\ref{le:TconstrainedCrossing}}
\label{fi:TconstrainedCrossing}
\end{figure}
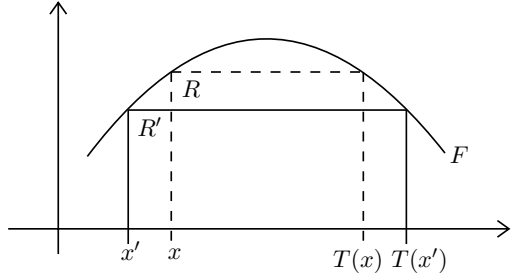
  To see that $T(x')\geq T(x)$, it suffices to check that the segment $[x',T(x)]\times \{F(x')\}$ is contained in $H$. In view of $T(x')\geq x$, the first $(x-x')$-long part of the segment has that property, and the rest of the segment is contained in $R$ and thus in $H$. 
\end{proof}

We now have all the ingredients for the main result on $T$.

\begin{proof}[Proof of Theorem~\ref{th:main-transport}]
  In view of Lemma~\ref{le:Tcouples} and $T(x)\geq x$, we have that $P:=\mu\otimes \delta_{T}\in\cD(\mu,\nu)$. Lemma~\ref{le:TconstrainedCrossing} shows that $P$ is supported on a set with the constrained crossing property %
  and then Theorem~\ref{th:main} yields $P=P_{*}$.
\end{proof}

\subsection{Reduction of Atoms}
\label{se:atoms}

Let $\mu\lst\nu$ satisfy $\mu\wedge\nu=0$ as before, but consider the case where $\mu$ and $\nu$ may have atoms. We still write $F=F_\mu-F_{\nu}$, now this function is right-continuous rather than continuous. The idea is to reduce to the atomless case by a transformation which inserts an interval at the location of each atom, with its length corresponding to the atom's mass. The atom is then replaced by a uniform density (cf.\ Figure~\ref{fi:jumpTransform}).

\begin{figure}[tbh]
\begin{center}
\resizebox{1.01 \textwidth}{!}{
\tikzset{every picture/.style={line width=0.75pt}} %
\begin{tikzpicture}[x=0.75pt,y=0.75pt,yscale=-1,xscale=1]
\draw  (15,215.4) -- (303.5,215.4)(43.85,93) -- (43.85,229) (296.5,210.4) -- (303.5,215.4) -- (296.5,220.4) (38.85,100) -- (43.85,93) -- (48.85,100)  ;
\draw    (66.5,204) .. controls (77,199.6) and (91.67,196.67) .. (97.67,178.67) ;
\draw    (133.67,118) .. controls (168,118) and (215.5,198) .. (276.5,203) ;
\draw  [dash pattern={on 0.84pt off 2.51pt}]  (97.83,178.67) -- (97.83,220.33) ;
\draw  [dash pattern={on 0.84pt off 2.51pt}]  (97.83,178.67) -- (220.08,178.67) ;
\draw    (98.17,147.33) .. controls (101.33,134.67) and (107,118) .. (133.67,118) ;
\draw  [dash pattern={on 0.84pt off 2.51pt}]  (98.17,147.33) -- (181.67,147.33) ;
\draw  (332,215.4) -- (620.5,215.4)(360.85,93) -- (360.85,229) (613.5,210.4) -- (620.5,215.4) -- (613.5,220.4) (355.85,100) -- (360.85,93) -- (365.85,100)  ;
\draw    (383.5,204) .. controls (394,199.6) and (408.67,196.67) .. (414.67,178.67) ;
\draw    (481.67,118) .. controls (516,118) and (563.5,198) .. (624.5,203) ;
\draw  [dash pattern={on 0.84pt off 2.51pt}]  (414.83,178.67) -- (414.83,220.33) ;
\draw  [dash pattern={on 0.84pt off 2.51pt}]  (414.83,178.67) -- (567.5,178.67) ;
\draw    (446.17,147.33) .. controls (449.33,134.67) and (455,118) .. (481.67,118) ;
\draw  [dash pattern={on 0.84pt off 2.51pt}]  (446.17,147.33) -- (529.67,147.33) ;
\draw [line width=1.5]    (414.83,178.67) -- (446.17,147.33) ;
\draw  [dash pattern={on 0.84pt off 2.51pt}]  (445.83,147.33) -- (445.83,220.33) ;
\draw  [dash pattern={on 0.84pt off 2.51pt}]  (530.83,147.33) -- (530.83,220.33) ;
\draw  [dash pattern={on 0.84pt off 2.51pt}]  (567.83,178.67) -- (567.83,220.33) ;
\draw   (276.5,136.25) -- (312.2,136.25) -- (312.2,127) -- (336,145.5) -- (312.2,164) -- (312.2,154.75) -- (276.5,154.75) -- cycle ;
\draw  [fill={rgb, 255:red, 0; green, 0; blue, 0 }  ,fill opacity=1 ] (415.35,178.67) .. controls (415.35,178.38) and (415.12,178.15) .. (414.83,178.15) .. controls (414.55,178.15) and (414.31,178.38) .. (414.31,178.67) .. controls (414.31,178.95) and (414.55,179.19) .. (414.83,179.19) .. controls (415.12,179.19) and (415.35,178.95) .. (415.35,178.67) -- cycle ;
\draw  [fill={rgb, 255:red, 0; green, 0; blue, 0 }  ,fill opacity=1 ] (446.69,147.33) .. controls (446.69,147.05) and (446.45,146.81) .. (446.17,146.81) .. controls (445.88,146.81) and (445.65,147.05) .. (445.65,147.33) .. controls (445.65,147.62) and (445.88,147.85) .. (446.17,147.85) .. controls (446.45,147.85) and (446.69,147.62) .. (446.69,147.33) -- cycle ;
\draw (99,230) node    {$x$};
\draw (313,215) node    {$x$};
\draw (47,80) node    {$F( x)$};
\draw (414,230) node    {$j( x-)$};
\draw (630,215) node    {$z$};
\draw (364,80) node    {$F'( z)$};
\draw (451,230) node    {$j( x)$};
\draw (525,230) node    {$T'( j( x))$};
\draw (588,230) node    {$T'( j( x-))$};
\end{tikzpicture}
}
\end{center}
\vspace{-1.5em}\caption{Transformation of an atom in $\mu$ at $x$.}
\label{fi:jumpTransform}
\end{figure}
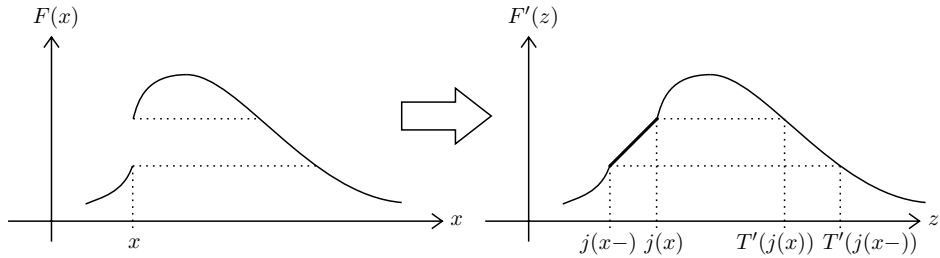

Let $\tau=\mu+\nu$ be the total variation and let
$$
  j(x)=x + \sum_{y\le x} |F(y) -F(y-)|,\quad x\in\R
$$
be the sum of the identity function and the cdf of the jump part of~$\tau$. Clearly~$j$ is strictly increasing and right-continuous; we denote its right-continuous inverse function by $j^{-1}: j(\R)\to\R$. Moreover, let
$$
  J_{x}=[j(x-),j(x)]
$$
be the interval representing the jump of $j$ at $x$. In particular, $J_{x}$ is an interval of length $\tau(\{x\})$ and a singleton $\{j(x)\}$ if $x$ is not an atom of $\mu$ or $\nu$.

Define an auxiliary measure $\mu'$ on $\R$ through its cdf as follows: for $z\in j(\R)$ we set $F_{\mu'}(z)=F_{\mu}(j^{-1}(z))$, whereas on the complement of $j(\R)$ we define $F_{\mu'}(z)$ by linearly interpolating from its values on $j(\R)$. In other words, $\mu'$ is defined by the two properties that $F_{\mu'} ( j(x) )= F_{\mu}(x)$ for $x\in \R$ and if $\tau$ has an atom at $x$, then $\mu'$ is uniform on the interval $J_{x}$ with total mass $\mu'(J_{x})=\mu(\{x\})$. It follows that $j$ is measure-preserving in the sense that $\mu'(j(B))=\mu(B)$ for any $B\in\cB(\R)$.
A second measure $\nu'$ is defined analogously from $\nu$. 

The construction implies that $\mu'\lst\nu'$  if and only if  $\mu\lst\nu$, and   $\mu'\wedge\nu'=0$ if and only if $\mu\wedge\nu=0$. Moreover, $\mu'$ and $\nu'$ are atomless. Thus, Theorem~\ref{th:main-transport} applies to $F'=F_{\mu'}-F_{\nu'}$ and yields a Monge map $T':=T(\mu',\nu')$. Reversing the transformation $j$, this map describes the desired coupling $P_{*}(\mu,\nu)$ as follows. (Of course, we can further apply Proposition~\ref{pr:decompose} to produce a statement analogous to Corollary~\ref{co:idAndT}, covering the case of arbitrary marginals $\mu\lst\nu$ without imposing the condition $\mu\wedge\nu=0$.)

\begin{theorem}\label{th:atomsTransform}
Let $\mu\wedge\nu=0$ and define $T'=T(\mu',\nu')$ as above. Then $P_{*}(\mu,\nu)=\mu\otimes \kappa$ for the stochastic kernel
$$
  \kappa(x) = \begin{cases}
   
  \frac{1}{\mu(\{x\})} \nu(\,\cdot \, \cap j^{-1} (T' (J_x) ))  & \mbox{if}~~\mu(\{x\})>0, \\
  \delta_{j^{-1}(T'(j(x)))} & \mbox{if}~~ \mu(\{x\})=0.
  \end{cases} 
$$
In particular, $\kappa$ is of Monge-type with transport map $T(x)=j^{-1}(T'(j(x)))$ whenever $\mu$ is atomless.
\end{theorem}

\begin{proof}
  If $\mu(\{x\})>0$, then $\kappa(x)$ is well defined by Lemma~\ref{le:Tpositive} and has the proper normalization as
  $\mu(\{x\}) = \mu'(J_x) = \nu'(T'(J_x))$. Among the points $x$ with $\mu(\{x\})=0$, it suffices to consider those with $j(x)\in I'$, the set of points of strict increase of $F'$---indeed, as $j$ is measure-preserving, it follows from Lemma~\ref{le:Tpositive} that the complementary set is $\mu$-null. For $j(x)\in I'$, Lemma~\ref{le:Tpositive} shows that $\kappa(x)=\delta_{j^{-1}(T'(j(x)))}$ is well defined.
  As $T'$ defines a coupling in $\cD(\mu',\nu')$ and $j$ is strictly monotone and measure-preserving, it follows that $\kappa$ defines a coupling in $\cD(\mu,\nu)$. Moreover, we know that the graph $\Gamma'$ of $T'$ has the constrained crossing property (Lemma~\ref{le:TconstrainedCrossing}). The strictly monotone transform $j$ does not invalidate that property (Corollary~\ref{co:invariance}), hence  $\Gamma:=j^{-1}(\Gamma')$ has the same property, and $\Gamma$ carries $\mu\otimes \kappa$, as noted above. We conclude by Theorem~\ref{th:main}.   
\end{proof} 

We note that $P_{*}$ can still be of Monge-type when $\mu$ has atoms: by Theorem~\ref{th:atomsTransform}, that happens precisely if $j^{-1} (T' (J_x) )$ is a singleton whenever $\mu(\{x\})>0$. This requires very specific atoms in $\nu$, as $\kappa$ must transport each 
upward jump point of $F$ to a downward jump point, and moreover the downward jump must have at least the same size as the upward jump. One example of such a match-up is given in (a) below.

\begin{example}[Empirical Distributions]
Consider marginals $\mu=\frac{1}{n_{\mu}} \sum_{i=1}^{n_{\mu}} \delta_{x_{i}}$ and $\nu =\frac{1}{n_{\nu}} \sum_{i=1}^{n_{\nu}} \delta_{y_{i}}$ in stochastic order.

(a) If the $x_{i}$ are distinct and $n_{\mu}=n_{\nu}=:n$, then $P_{*}$ is Monge and the transport map $T$ is as constructed in the introduction: considering the destinations $S_{1}=\{y_{1},\dots,y_{n}\}$ as a multi-set (i.e., distinguishing the $y_{i}$ even if they have the same value), we iterate for $k=1,\dots,n$:
\begin{enumerate} 
\item $T(x_{k}):= \min \{y\in S_{k}:\, y\geq x_{k}\}$,
\item $S_{k+1}= S_{k}\setminus \{T(x_k)\}$.
\end{enumerate} 

(b) The case $n_{\mu}\neq n_{\nu}$ is natural when $\mu$ and $\nu$ are empirical distributions of observed data---in the study of treatment effects, data are  often not observed in pairs and hence the two marginals may not have the same number of observations; see Section~\ref{se:introduction}.  The above algorithm immediately extends to the case where $n_{\mu}=m n_{\nu}$ for an integer $m$, by redefining the $y_{i}$. If $n_{\mu}$ and $n_{\nu}$ are arbitrary, and/or the atoms have possibly different, rational weights, we can still write the marginals in the form $\mu=\frac{1}{n} \sum_{i=1}^{n} \delta_{x_{i}}$ and $\nu =\frac{1}{n} \sum_{i=1}^{n} \delta_{y_{i}}$ after by choosing a suitable $n$, now with the $x_{i}$ not necessarily distinct. The principle of the above algorithm to find $P_{*}$ still applies, but when several $x_{i}$ are at the same location, it will typically deliver a randomized coupling since an atoms of $\mu$ may be mapped into multiple atoms of $\nu$.
\end{example}

\section{Further Properties}\label{se:furtherProperties}
\subsection{Antitone Decomposition}\label{se:antitone}
 
As seen in Example~\ref{ex:single-crossing}, $P_{*}$ is the sum of an identity coupling and an antitone coupling when the marginal densities satisfy 
a single-crossing condition. In this section, we analyze to which extent such a decomposition generalizes to other marginals. The first result (together with Proposition~\ref{pr:decompose}) shows that $P_{*}$ is always the sum of an identity coupling and countably many antitone couplings. We will see that in certain cases, the marginal measures for those antitone coupling are simply restrictions of $\mu$ and $\nu$ to specific intervals, as in the aforementioned example. 
In general, however, the decomposition remains more implicit as the marginal measures do not admit such a simple description.

\begin{proposition}\label{pr:antitone}
Let $\mu \lst \nu$ satisfy $\mu\wedge\nu=0$. Then $P_*$ is the sum of countably many antitone couplings.
\end{proposition}

\begin{proof}  
In view of Theorem~\ref{th:atomsTransform}, we may assume that $\mu,\nu$ are atomless.
For any continuous, nonnegative, nonconstant function $G$ of  finite variation  with $G(-\infty)=G(\infty)=0$, we define 
$x_G=\min(\argmax G)$ as the smallest global maximum point %
and set
$$
G' (x) = \min_{y\in [x,x_G]} G(y) \1_{\{x\le x_G\}} +\min_{y\in [x_G,x]} G(y) \1_{\{x> x_G\}}, 
$$
whereas if $G\equiv 0$, we use $x_{G}:=-\infty$ instead.
Note that $G'$ is continuous, increasing on $(-\infty,x_{G}]$ and decreasing on $[x_{G},\infty)$, with $0\le G'\le G$ and $\max G'=\max G$. Thus
$G'$ can be decomposed as 
$G'=F_{\mu'} -F_{\nu'}$ where the singular measures $\mu'$ and $\nu'$ can be coupled by a directional antitone coupling. %
This coupling, while equal to $P_{*}(\mu',\nu')$, will be denoted by $P(G)$ for brevity. Moreover, $\mu'\leq \mu$ and $\nu'\leq \nu$. Finally, the total variation $V(G')=(\mu'+\nu')(\R)$ satisfies $V(G')\geq 2\max G'=2\max G$.

Define $F_1:=F$ and
$$
F_{k+1} := F_{k} - F'_{k},\quad k\geq1.
$$
Using the above notation,
$P(F_k)$  is the directional antitone coupling between the singular measures $\mu_{k}',\nu_{k}'$ forming a decomposition for $F'_k$. 

To see that $F=\sum_{k} F_k'$, note that $V(F_{k}')\to0$ as $\sum_{k}  V(F_k') \le V(F)=2$. On the other hand, $V(F'_{k})\geq 2\max F_{k}$, so that $\max F_{k}\to0$; that is, $F_{k}$ uniformly decreases to zero and in particular $F=\sum_{k} F_k'$. This shows that $\sum_{k} P(F_k)$ is a coupling of $\mu$ and $\nu$. Clearly this coupling is directional, and thus equal to $P_{*}(\mu,\nu)$ by Theorem~\ref{th:main} if it satisfies the constrained crossing property.
To verify the latter, let $x$ be a point of strict increase of $F_{k}$ and suppose that the transport map $T_{k}$ of $P(F_{k})$ maps $x$ to $y$. Then $F_{k}(x)=F_{k}(y)$ and $F_{k}(z)\geq F_{k}(x)>0$ for all $z\in [x,y]$. It follows for any $j<n$ that $F'_{j}(z)< F_{j}(z)$ for all $z\in [x,y]$, which in turn implies that $F'_{j}$ is constant over the interval $[x,y]$. In other words, the couplings $P(F_{j})$ for $j<k$ cannot transport any mass into the interval or out of the interval. This shows the constrained crossing property, and in addition that the marginals $\mu'_{j}$ (resp.~$\nu'_{j}$) of $P(F_{j})$, $j\leq k$ are supported on disjoint sets which are finite unions of intervals.
\end{proof}

In particular cases, we can obtain the antitone couplings in $P_*$ explicitly as antitone couplings between disjoint intervals.

\begin{example}[Multiple-crossing Densities]
  Assume that $\mu$ and $\nu$ are atomless and that $F=F_{\mu}-F_{\nu}$ is piecewise monotone (with finitely many pieces). Then by inspecting the proof of Proposition~\ref{pr:antitone}, we see that
$P_{*} $ is the sum of the identical coupling of $\mu\wedge \nu$ and finitely many antitone couplings between pairs of \emph{disjoint} intervals.

As an important special case extending Example~\ref{ex:single-crossing}, suppose that $\mu$ and $\nu$ have continuous densities that cross finitely many times. Then $F=F_{\mu}-F_{\nu} = F_{\mu-\mu\wedge \nu} -F_{\nu-\mu\wedge \nu}$ is piecewise monotone and the optimal coupling between $\mu-\mu\wedge \nu$ and $\nu-\mu\wedge \nu$ is the sum of finitely many antitone couplings between disjoint intervals.  %
\end{example} 

In contrast to the above example, the following shows that a decomposition into antitone couplings between intervals is not possible in general.

\begin{example}[Absence of Antitone Intervals]
Let $\mu$ be the Cantor distribution  on $[0,1]$ %
and $\nu$ be uniform on $[0,2]$. Clearly $\mu\wedge \nu=0$. 
We first verify that  $\mu\lst \nu$, or equivalently $\cD(\mu,\nu)\neq\emptyset$.
Each element $x\in C$ can be represented  in base $3$
as
$x=2 \sum_{n=1}^\infty  { x_n}{3^{-n}}$
where $x_n\in\{0,1\}$.
The comonotone transport $T_C$
given by $T_C(x)=2 \sum_{n=1}^\infty x_{ n}{2^{-n}}$
is directional and
transports $\mu$ to $\nu$.
Hence, $\mu\lst \nu$.  

Next, we show that $P_{*}\in \cD(\mu,\nu)$ does not contain any
 antitone couplings between intervals.
Assume for contradiction that 
there exists an interval $[a,b]\subseteq [0,1]$
such that $\mu([a,b])>0$ and $T|_{[a,b]}$ is the antitone mapping between $\mu|_{[a,b]}$ and its image.
This implies that there exists $c$ such that  $\mu([a,c])>0$ and 
$T$ transports $\mu|_{[a,c]}$ to a distribution supported by $(c,\infty)$.
However, by Theorem \ref{th:existenceAndMinimality},
$T$ transports $\mu|_{(a,\infty)}$ to a distribution $\nu_a$ whose minimality property together with $\nu([a,c])>0$ imply that $\nu_{a}$ charges $[a,c]$, a contradiction.
\end{example}

\subsection{Optimality as Unconstrained Transport}\label{se:unconstrained}
 
The optimal directional coupling $P_{*}$ is also the optimizer for certain classical transport problems (unconstrained and with finite cost function) where the constraint is ``not binding,'' although only for specific marginals. We confine ourselves to giving one example.
Consider $\mu\lst\nu$ and the transport problem 
 \begin{equation}
\inf_{P} \int c(|y-x|) \, P(dx,dy)
\label{eq:unconstrainedProblemEx}
\end{equation}
over all couplings~$P$ of $\mu$ and $\nu$. Suppose that $c:\R\to\R_{+}$ is increasing and concave, so that $c(|y-x|)$ is supermodular on $\H$ but (typically) not on $\R^{2}$.%

\begin{proposition}\label{pr:unconstrained}
If $F=F_\mu-F_{\nu}$ is unimodal, then $P_*(\mu,\nu)$ is an optimal coupling for the unconstrained problem~\eqref{eq:unconstrainedProblemEx}. If $c$ is strictly concave, the optimizer is unique.
\end{proposition}

This follows from the general results stated in \cite[Part~II]{GangboMcCann.96}. A direct argument is sketched below.

\begin{proof}
  We know from Theorem~\ref{th:main} that $P_{*}$ is optimal among all directional couplings. To rule out that a non-directional coupling has a smaller cost, the key observation is that if $P$ is an optimizer, it is concentrated on a $c$-cyclically monotone set $\Gamma$, which implies that $\Gamma$ cannot contain pairs $(x,y),(x',y')$ with $y<x$ and either (i) $x'\in [y,x)$ and $y'\geq y$ or (ii) $y'\in[x,y)$ and $x'\leq x$. Together with the unimodality condition, this can be seen to imply the result. We omit the details in the interest of brevity.
\end{proof} 

The unimodality condition in Proposition~\ref{pr:unconstrained} is crucial; e.g., the assertion fails for $\mu=\frac12(\delta_{0}+\delta_{13})$ and $\nu=\frac12(\delta_{12}+\delta_{25})$ with cost function $\sqrt{|y-x|}$.

\subsection{Other Constraints}\label{se:moreConstraints}

The directional constraint $Y\geq X$ naturally generalizes to  
$Y \geq X + D$ for a measurable function $D: \R \to \R$ such that $x\mapsto x + D(x)$ is strictly increasing. For instance, if $D\equiv d$ is constant, this means that the transport must travel as least a distance $d$ to the right (or at most distance $|d|$ to the left, if $d<0$). While $Y\geq X$ is equivalent to $P(\H)=1$, the generalized constraint is expressed as $P(\D)=1$ for the epigraph~$\D$ of $x\mapsto x + D(x)$. We denote by $\cD_{D}(\mu,\nu)$ the set of all such couplings~$P$ of $\mu,\nu$.

The construction of $P_{*}$ naturally extends to this constraint. %
Indeed, let $Z(x)=x+D(x)$ and consider arbitrary distributions~$\mu$ and~$\nu$ on~$\R$. We define the transformed marginal $\mu'=\mu\circ Z^{-1}$ and define $\mu\preceq_{D}\nu$ to mean that $\mu'\lst\nu$. Then $\mu\preceq_{D}\nu$ if and only if $\cD_{D}(\mu,\nu)\neq\emptyset$, and more generally, the transformation $Z$ induces a bijection between $\cD_{D}(\mu,\nu)$ and the set $\cD(\mu',\nu)$ of directional couplings between~$\mu'$ and~$\nu$. If we define the analogues of the constrained crossing property, constrained submodularity, etc., for $\D$, this bijection preserves the crossing/optimality properties and we find that
$$
  P_{*}^{D}(\mu,\nu) := P_{*}(\mu',\nu)\circ (Z,\id)
$$
has the properties analogous to the optimal directional coupling for the constraint~$\D$. We omit the details in the interest of brevity.

\bibliography{stochfin}
\bibliographystyle{plain}

\end{document}